\documentclass[11pt]{amsart}

\usepackage[margin=1.5in]{geometry} 
\usepackage{amsmath,amsthm,amssymb,bbm,bm}
\usepackage{graphicx, xcolor}
\usepackage{hyperref}
\usepackage{amssymb,amsfonts}
\usepackage[all,arc]{xy}
\usepackage{enumerate}
\usepackage{mathrsfs}
\usepackage{xcolor}
\usepackage{enumerate, enumitem}

\newcommand{\N}{\mathbb{N}}

\newcommand{\R}{\mathbb{R}}

\newcommand{\F}{\mathbb{F}}

\newcommand{\GL}{\mathrm{GL}}
\newcommand{\Tc}{\mathcal{T}}
\newcommand{\I}{\mathbbm{1}}
\newcommand{\Var}{\mathrm{Var}}
\newcommand{\cov}{\mathrm{cov}}
\newcommand{\Poi}{\mathtt{Poisson}}

\newcommand{\des}{\mathrm{des}}
\newcommand{\inv}{\mathrm{inv}}
\newcommand{\fp}{\mathrm{fp}}
\newcommand{\up}{\mathrm{up}}

\newcommand\bslash{\char`\\}

\newcommand{\Address}{{
\bigskip
\footnotesize

\textsc{Department of Mathematics, University of Southern California, Los Angeles, CA}\par\nopagebreak
\textit{E-mail address}: \texttt{paguyo@usc.edu}
}}

\def\bal#1\eal{\begin{align*}#1\end{align*}}

\newtheorem{theorem}{Theorem}[section]
\newtheorem{lemma}[theorem]{Lemma}
\newtheorem{proposition}[theorem]{Proposition}
\newtheorem{corollary}[theorem]{Corollary}

\title[Fixed points, descents, \& inversions in double cosets of $S_n$]{Fixed points, descents, and inversions in parabolic double cosets of the symmetric group} 
\author{J. E. Paguyo}


\begin{document}


\begin{abstract}
We consider statistics on permutations chosen uniformly at random from fixed parabolic double cosets of the symmetric group. 
Under some mild conditions, we show that the distribution of fixed points is asymptotically Poisson and establish central limit theorems for the distribution of descents and inversions. 
Our proofs use Stein's method with size-bias coupling and dependency graphs, which also give convergence rates for our distributional approximations.
As applications of our size-bias coupling and dependency graph constructions, we also obtain concentration of measure results. 
\end{abstract}

\maketitle



\section{Introduction}

Let $H$ and $K$ be subgroups of a finite group $G$. Define an equivalence relation on $G$ by
\bal
s \sim t \iff s = htk, \qquad \text{for $s,t \in G$, $h \in H$, $k \in K$}.
\eal
The equivalence classes are the {\em double cosets} of $G$. Let $HsK$ denote the double coset containing the element $s$ and let $H \bslash G / K$ denote the set of double cosets. 

Let $\lambda = (\lambda_1, \ldots, \lambda_I)$ be a {\em partition} of $n$, so that $\lambda_1 \geq \dotsb \geq \lambda_I > 0$ and $\lambda_1 + \dotsb + \lambda_I = n$. Let $\ell(\lambda)$ be the length of $\lambda$. 
The {\em parabolic subgroup}, $S_\lambda$, is the set of permutations in $S_n$ that permute $\{1,\ldots,\lambda_1\}$ among themselves, $\{\lambda_1 + 1, \ldots, \lambda_1 + \lambda_2\}$ among themselves, and so on. 
Thus we can write $S_\lambda \cong S_{\lambda_1} \times \dotsb \times S_{\lambda_I}$. 

Let $\mu = (\mu_1,\ldots, \mu_J)$ be another partition of $n$. Then $S_\lambda \bslash S_n / S_\mu$ are the {\em parabolic double cosets} of $S_n$. 
Let $\Tc_{\lambda, \mu}$ be the set of $I \times J$ {\em contingency tables} with nonnegative integer entries whose row sums equal $\lambda$ and column sums equal $\mu$. 
Then the parabolic double cosets $S_\lambda \bslash S_n / S_\mu$ are in bijection with $\Tc_{\lambda, \mu}$,

The bijection can be described as follows. For the partition $\lambda$, define the {\em $\lambda_i$-block} to be the set of elements 
$L_i  := \{\lambda_1 + \dotsb + \lambda_{i-1}+1, \lambda_1 + \dotsb + \lambda_{i-1} + 2, \ldots, \lambda_1 + \dotsb + \lambda_i\}$, for $1 \leq i \leq \ell(\lambda)$, where $\lambda_0$ is defined to be zero. 
Define the {\em $\mu_i$-block}, $M_i$, analogously for the partition $\mu$. 
Fix $\sigma \in S_n$. Then its corresponding contingency table $T = T(\sigma) = (T_{ij})$ is defined to be the table such that entry $T_{ij}$ is the number of elements from the set $M_j$ which occur in the positions of $\sigma$ given by the set $L_i$. 
This gives a map from $S_n$ to $\Tc_{\lambda, \mu}$. 

Thus we may write $T = T(\sigma)$ as the contingency table representing the double coset $S_\lambda \sigma S_\mu$, 
and we say that the double coset $S_\lambda \sigma S_\mu$ is {\em indexed} by $T$. 

The uniform distribution on $S_n$ induces the {\em Fisher-Yates distribution} on contingency tables, given by
\bal
P_{\lambda, \mu}(T) = \frac{1}{n!} \frac{(\prod_i \lambda_i!) (\prod_j \mu_j!)}{\prod_{i,j} T_{ij}!}.
\eal
It follows that the size of the double coset indexed by $T$ is given by
\bal
|S_\lambda \sigma S_\mu| = \frac{(\prod_i \lambda_i!) (\prod_j \mu_j!)}{\prod_{i,j} T_{ij}!}. 
\eal
We remark that enumerating parabolic double cosets remains a difficult problem; in fact, it is $\#P$-complete \cite{DG95}. 

Contingency tables are a mainstay in statistics and there has been extensive work on their statistical properties. 
The importance of contingency tables can be seen in the fact that they arise when a population of size $n$ is classified according to two discrete categories. 
Ronald Fisher studied the distribution of $I \times J$ contingency tables conditioned on fixed row and column sums, $\lambda$ and $\mu$, respectively. 
Then the conditional probability $P(T \mid \lambda, \mu)$ of obtaining the contingency table $T$ is precisely the Fisher-Yates distribution. 

Under the Fisher-Yates distribution, Kang and Klotz \cite{KK98} showed that the joint distribution of the entries of the table $T = (T_{ij})$ is asymptotically multivariate normal. 
As corollaries, they also showed that each entry $T_{ij}$ is asymptotically normal, and that the chi-squared statistic
\bal
\chi^2(T) = \sum_{i,j} \frac{\left( T_{ij} - \frac{\lambda_i \mu_j}{n} \right)^2}{\frac{\lambda_i \mu_j}{n}},
\eal
which measures how ``close" a table is to a product measure on tables, is asymptotically chi-squared distributed with $(I-1)(J-1)$ degrees of freedom \cite{KK98}. 
This shows that a typical contingency table looks like the {\em independence table}, $T^* = (T_{ij}^*) = \left(\frac{\lambda_i \mu_j}{n}\right)$. Moreover, it follows that the largest double cosets are those which are closest to the independence table. 
Barvinok \cite{Bar10} obtained analogous results for uniformly distributed contingency tables.  

Contingency tables can be partially ordered as follows. Let $T$ and $T'$ be contingency tables with the same row and column sums. 
Then $T'$ {\em majorizes} $T$, denoted by $T \prec T'$, if the largest element in $T'$ is greater than the largest element in $T$, the sum of the two largest elements in $T'$ is greater than the sum of the two largest elements in $T$, etc. 
Let $T,T'$ be Fisher-Yates distributed contingency tables. It follows that if $T \prec T'$, then $P(T) > P(T')$; see \cite{DS21} for a short proof. 
Joe \cite{Joe85} proved that the independence table $T^*$ is the unique smallest table in majorization order among all contingency tables with the same row and column sums. 

A natural statistic on contingency tables is the number of zero entries. Recall that most tables are close to the independence table $T^*$, which contains no zero entries, and so we expect most contingency tables to have no zero entries. 
Diaconis and Simper \cite{DS21} showed, under some mild hypotheses, that the number of zeros in a Fisher-Yates distributed contingency table is asymptotically Poisson. 
Their proof uses the observation that a Fisher-Yates distributed contingency table is equivalent to rows of independent multinomial vectors, conditioned on the column sums. 
As open problems, they asked for the distribution of the positions of the zeros, and for the size and distribution of the largest entry in a Fisher-Yates distributed contingency table.

Another line of work is to consider double cosets of other groups $G$ and subgroups $H,K$ and study the distribution that a uniform distribution of $G$ induces on $H \bslash G / K$. The following two examples are treated in detail in \cite{DS21}. 

Let $G = \GL_n(\F_q)$, where $\F_q$ is a finite field of size $q$, and let $H = K$ be the subgroup of lower triangular matrices. 
Then the double cosets $H \bslash G / K$ are indexed by permutations, and the uniform distribution on $G$ induces the {\em Mallows measure} on $S_n$,
\bal
p_q(\sigma) = \frac{q^{\inv(\sigma)}}{[n]_q!},
\eal
where $\inv(\sigma)$ denotes the number of inversions of the permutation $\sigma$ and $[n]_q! = \prod_{i=1}^n (1 + q + \dotsb + q^{i-1})$. 

Let $G = S_{2n}$ be the symmetric group on $[2n]$, and let $H = K$ be the hyperoctahedral subgroup of centrally symmetric permutations, which are isomorphic to $C_2^n \rtimes S_n$, the group of symmetries of an $n$-dimensional hypercube. 
Then the double cosets are indexed by the partitions of $n$, and the uniform distribution on $G$ induces the {\em Ewens's sampling formula}, 
\bal
p_q(\lambda) = \frac{n! q^{\ell(\lambda)}}{z \cdot z_\lambda},
\eal
where $\ell(\lambda)$ is the length of the partition $\lambda$, $z_\lambda = \prod_{i=1}^n i^{a_i} a_i!$, and $z = q(q+1)\dots (q+n-1)$, with $a_i$ the number of parts of $\lambda$ of size $i$. 

As Diaconis and Simper remarked in \cite{DS21}, enumeration by double cosets can lead to interesting mathematics. 


\subsection{Main Results}

Let $\mu$ and $\nu$ be probability distributions. The {\em total variation distance} between $\mu$ and $\nu$ is
\bal
d_{TV}(\mu, \nu) := \sup_{A \subseteq \Omega} |\mu(A) - \nu(A)|, 
\eal
where $\Omega$ is a measurable space. The {\em Kolmogorov distance} between $\mu$ and $\nu$ is 
\bal
d_K(\mu, \nu) := \sup_{x\in \R} |\mu(-\infty, x] - \nu(-\infty, x]|.
\eal
If $X$ and $Y$ are random variables with distributions $\mu$ and $\nu$, respectively, then we write $d_K(X,Y)$ to denote the Kolmogorov distance between the distributions of $X$ and $Y$, and similarly for $d_{TV}$. 

We initiate the study of statistics on permutations chosen uniformly at random from a fixed parabolic double coset in $S_\lambda \bslash S_n / S_\mu$. 
This line of study was suggested by Diaconis and Simper in \cite{DS21}, as a parallel to recent work on permutations from fixed conjugacy classes of $S_n$ \cite{Ful98, Kim19, KL20, KL21, FKL19}. 
Under some mild conditions, we obtain limit theorems, with rates of convergence, for the number of fixed points, descents, and inversions. 

Let $\sigma \in S_n$. Then $\sigma$ has a {\em fixed point} at $i \in [n]$ if $\sigma(i) = i$. 
Let $\fp(\sigma)$ be the number of fixed points of $\sigma$. Then we can decompose $\fp(\sigma)$ into a sum of indicator random variables as 
$\fp(\sigma) = \sum_{i = 1}^n \theta_i$, where $\theta_i = \I_{\{\sigma(i) = i\}}$ is the indicator random variable that $\sigma$ has a fixed point at index $i$. Define $A_{k\ell} := L_k \cap M_\ell \subseteq [n]$ for all $1 \leq k \leq I$ and $1 \leq \ell \leq J$. 
Note that $A_{k\ell}$ may possibly be empty and that the sets $\{ A_{k\ell} \}$ form a disjoint partition of $[n]$. 

Our first result shows, under some hypotheses on the partitions $\lambda,\mu$ and the contingency table $T$, 
that the number of fixed points of a permutation chosen uniformly at random from a fixed parabolic double coset indexed by $T$ is asymptotically Poisson.
Moreover, we obtain an explicit error bound on the distributional approximation. 

\begin{theorem}\label{PoissonLimitTheorem} 
Let $\lambda = (\lambda_1,\ldots, \lambda_I)$ and $\mu = (\mu_1, \ldots, \mu_J)$ be two partitions of $n$. 
Let $\sigma \in S_n$ be a permutation chosen uniformly at random from a fixed double coset in $S_\lambda \bslash S_n / S_\mu$ indexed by $T$. 
Suppose $T_{k\ell} \geq 1$ for all $1 \leq k \leq I, 1 \leq \ell \leq J$ such that $A_{k\ell} \neq \emptyset$. 
Let $Y_n$ be a Poisson random variable with rate $\nu_n := E(\fp(\sigma))$. Then 
\bal
d_{TV}(\fp(\sigma), Y_n) \leq \frac{5(I+J-1) \min\{1, \nu_n^{-1}\}}{\max\{\lambda_I, \mu_J\}}. 
\eal

If the partitions $\lambda$ and $\mu$ also satisfy
\begin{enumerate}
\item $\displaystyle{\lim_{n \to \infty} \lambda_I} = \infty$ and $\displaystyle{\lim_{n \to \infty} \mu_J} = \infty$, \\
\item $\displaystyle{\lim_{n \to \infty} (I + J - 1) = C}$, for some constant $C \in \N$. 
\end{enumerate}
and if $\limsup_{n \to \infty} \nu_n = \nu < \infty$, then $d_{TV}(\fp(\sigma), Y) \to 0$ as $n \to \infty$, where $Y$ is a Poisson random variable with rate $\nu$, so that $\fp(\sigma) \xrightarrow{d} Y$. 
\end{theorem}

The proof uses size-bias coupling and Stein's method. The main advantage of using Stein's method is that we are able to obtain error bounds on our distributional approximations, in addition to proving limit theorems. 

We remark that Stein's method via size-bias coupling has been successfully used to prove Poisson limit theorems; some examples are:  
Angel, van der Hofstad, and Holmgren \cite{AHH19} for the number of self-loops and multiple edges in the configuration model, 
Arratia and DeSalvo \cite{AD17} for completely effective error bounds on Stirling numbers of the first and second kinds, 
Goldstein and Reinert \cite{GR06} for Poisson subset numbers, and 
Holmgren and Janson for sums of functions of fringe subtrees of random binary search trees and random recursive trees \cite{HJ15}.

The study of fixed points of random permutations began with Montmort \cite{Mon08}, who showed that the number of fixed points in a uniformly random permutation is asymptotically Poisson with rate $1$. 
Since then, the topic has branched out in many different directions. 
Diaconis, Fulman, and Guralnick \cite{DFG08} sought generalizations and obtained limit theorems for the number of fixed points under primitive actions of $S_n$.
In \cite{DEG14}, Diaconis, Evans, and Graham established a relationship between the distributions of fixed points and unseparated pairs in a uniformly random permutation. 
Mukherjee \cite{Muk16} used the recent theory of permutation limits to derive the asymptotic distributions of fixed points and cycle structures for several classes of non-uniform distributions on $S_n$, which include Mallows distributed permutations. 
More recently, Miner and Pak \cite{MP14} and Hoffman, Rizzolo, and Slivken \cite{HRS17} studied fixed points in random pattern-avoiding permutations and showed connections with Brownian excursion. 

Let $\sigma \in S_n$. A pair $(i,j) \subseteq [n]^2$ is a {\em $d$-descent}, or {\em generalized descent}, of $\sigma$ if $i < j \leq i + d$ and $\sigma(i) > \sigma(j)$. 
In particular, a {\em descent} corresponds to a $1$-descent and an {\em inversion} corresponds to an $(n-1)$-descent. 
Let $\des_d(\sigma)$ be the number of $d$-descents, $\des(\sigma)$ the number of descents, and $\inv(\sigma)$ the number of inversions in $\sigma \in S_n$. 

Our next three results establish central limit theorems for the number of descents, the number of $d$-descents for fixed $d$, and the number of inversions in a permutation chosen uniformly at random from a fixed double coset indexed by $T$ 
by providing upper bounds on the Kolmogorov distances to the standard normal of order $n^{-1/2}$.

\begin{theorem}\label{CLTdescents}
Let $\lambda = (\lambda_1,\ldots, \lambda_I)$ and $\mu = (\mu_1, \ldots, \mu_J)$ be two partitions of $n$ such that $I = o(n)$. 
Let $\sigma \in S_n$ be a permutation chosen uniformly at random from a fixed double coset in $S_\lambda \bslash S_n / S_\mu$. 
Let $W_n := \frac{\des(\sigma) - \mu_n}{\sigma_n}$, where $\mu_n := E(\des(\sigma))$ and $\sigma_n^2 := \Var(\des(\sigma))$. Then
\bal
d_K(W_n, Z) \leq O(n^{-1/2}),
\eal
where $Z$ is a standard normal random variable, so that $W_n \xrightarrow{d} Z$, as $n \to \infty$. 
\end{theorem}

\begin{theorem}\label{CLTddescents}
Let $\lambda = (\lambda_1,\ldots, \lambda_I)$ and $\mu = (\mu_1, \ldots, \mu_J)$ be two partitions of $n$ such that $I = o(n)$. 
Let $\sigma \in S_n$ be a permutation chosen uniformly at random from a fixed double coset in $S_\lambda \bslash S_n / S_\mu$. Let $d \leq \frac{\lambda_I}{2}$ be a fixed positive integer. 
Let $W_n := \frac{\des_d(\sigma) - \mu_n}{\sigma_n}$, where $\mu_n := E(\des_d(\sigma))$ and $\sigma_n^2 := \Var(\des_d(\sigma))$. Then
\bal
d_K(W_n, Z) \leq O(n^{-1/2}),
\eal
where $Z$ is a standard normal random variable, so that $W_n \xrightarrow{d} Z$, as $n \to \infty$. 
\end{theorem}

\begin{theorem}\label{CLTinversions}
Let $\lambda = (\lambda_1,\ldots, \lambda_I)$ and $\mu = (\mu_1, \ldots, \mu_J)$ be two partitions of $n$ such that $I = o(n)$. 
Let $\sigma \in S_n$ be a permutation chosen uniformly at random from a fixed double coset in $S_\lambda \bslash S_n / S_\mu$.  
Let $W_n := \frac{\inv(\sigma) - \mu_n}{\sigma_n}$, where $\mu_n := E(\inv(\sigma))$ and $\sigma_n^2 := \Var(\inv(\sigma))$. Then
\bal
d_K(W_n, Z) \leq O(n^{-1/2}),
\eal
where $Z$ is a standard normal random variable, so that $W_n \xrightarrow{d} Z$, as $n \to \infty$. 
\end{theorem}

The proofs use dependency graphs and Stein's method. Stein's method via dependency graphs is a powerful technique which reduces proving central limit theorems to a construction of a dependency graph and a variance estimate. 
This method has been successfully used to prove central limit theorems; some examples include 
Avram and Bertsimas \cite{AB93} for the dependence range of several geometric probability models where random points are generated according to a Poisson point process, 
Barany and Vu \cite{BV07} for the volume and number of faces of a Gaussian random polytope, and 
Hofer \cite{Hof18} for the number of occurrences of a fixed vincular permutation pattern in a random permutation. 

The probabilistic study of descents and inversions of random permutations is a rich and expansive topic. 
We do not attempt to cover the entire history, but we mention some previous work. 
There are many proofs for the asymptotic normality of descents, but we highlight Fulman's proof using Stein's method via exchangeable pairs and a non-reversible Markov chain in \cite{Ful04}. 
Bona \cite{Bon08} used the method of dependency graphs to show the asymptotic normality of generalized descents, and subsequently Pike \cite{Pik11} used Stein's method and exchangeable pairs to obtain rates of convergence. 
Chatterjee and Diaconis \cite{CD17} proved a central limit theorem for the two sided descent statistic in uniformly random permutations, and 
He \cite{He22} extended this result to Mallows distributed permutations. 
In another direction, Conger and Viswanath \cite{CV07} established the asymptotic normality of the number of descents and inversions in permutations of mulitsets. 

\subsection{Outline}

The paper is organized as follows. 
Section \ref{Preliminaries} introduces notation and definitions that we use throughout the paper, and gives necessary background and relevant results on size-bias coupling, dependency graphs, Stein's method, and concentration inequalities. 

In Section \ref{FixedPoints} we use size-bias coupling and Stein's method to prove Theorem \ref{PoissonLimitTheorem}, 
and in Sections \ref{Descents}, \ref{GeneralizedDescents}, and \ref{Inversions} we construct dependency graphs for descents and inversions, and apply Stein's method to prove Theorems \ref{CLTdescents}, \ref{CLTddescents}, and \ref{CLTinversions}. 

We then use our size-bias coupling and dependency graph constructions to prove concentration of measure results in Section \ref{Concentration}. 

We conclude with some final remarks in Section \ref{FinalRemarks}.


\section{Preliminaries} \label{Preliminaries}

\subsection{Notation and Definitions} \label{NotationDefinitions}

For a positive integer $n$, let $[n] := \{1,\ldots,n\}$. Let $S_n$ be the symmetric group on $[n]$. 

As an abuse of notation, we use $\sigma$ to denote a permutation, while $\sigma_n^2$ denotes the variance of some random variable, which may depend on $n$. 

Let $a_n, b_n$ be two sequences. If $\lim_{n \to \infty} \frac{a_n}{b_n}$ is a nonzero constant, then we write $a_n \asymp b_n$ and we say that $a_n$ is of the same {\em order} as $b_n$. 
If there exists positive constants $c$ and $n_0$ such that $a_n \leq cb_n$ for all $n \geq n_0$, then we write $a_n = O(b_n)$. 

Let $\lambda = (\lambda_1, \ldots, \lambda_N)$ be a partition of $n$. 
Recall from the introduction that the {\em $\lambda_i$-block} is the set $L_i := \{\lambda_1 + \dotsb + \lambda_{i-1}+1, \ldots, \lambda_1 + \dotsb + \lambda_i\}$. 
Define the {\em $\lambda_i$-border} be the value $l_i^b := \lambda_1 + \dotsb + \lambda_i$. 
Then the {\em $\lambda_i$-interior} is defined to be the set $L_i^o := L_i \setminus \{l_i^b\}$, that is, the $\lambda_i$-block minus the $\lambda_i$-border. 
Also recall from the introduction the sets $A_{k\ell} := L_k \cap M_\ell \subseteq [n]$ for all $1 \leq k \leq I$ and $1 \leq \ell \leq J$, and note that these sets $\{ A_{k\ell} \}$ form a disjoint partition of $[n]$. 

\subsection{Size-Bias Coupling}

Let $W$ be a non-negative integer-valued random variable with finite mean. Then $W^s$ has the {\em size-bias distribution} of $W$ if
\bal
P(W^s = w) = \frac{wP(W = w)}{EW}.
\eal

A {\em size-bias coupling} is a pair $(W, W^s)$ of random variables defined on the same probability space such that $W^s$ has the size-bias distribution of $W$. 

We use the following outline provided in \cite{Ros11} for coupling a random variable $W$ with its size-bias version $W^s$. 

Let $W = \sum_{i=1}^n X_i$ where $X_i \geq 0$ and $p_i = EX_i$. 
\begin{enumerate}
\item For each $i \in [n]$, let $X_i^s$ have the size-bias distribution of $X_i$ independent of $(X_k)_{k \neq i}$ and $(X_k^s)_{k \neq i}$. 
Given $X_i^s = x$, define the vector $(X_k^{(i)})_{k \neq i}$ to have the distribution of $(X_k)_{k \neq i}$ conditional on $X_i = x$. 
\item Choose a random summand $X_I$, where the index $I$ is chosen, independent of all else, with probability $P(I = i) = p_i/\lambda$, where $\lambda = EW$. 
\item Define $W^s = \sum_{k \neq I} X_k^{(I)} + X_I^s$. 
\end{enumerate}

If $W^s$ is constructed by Items (1)-(3) above, then $W^s$ has the size-bias distribution of $W$. As a special case, note that if the $X_i$ is an indicator random variable, then its size-bias distribution is $X_i^s = 1$. 
We summarize this as the following proposition. 

\begin{proposition}[\cite{Ros11}, Corollary 4.14] \label{sizebiascouplingconstruction}
Let $X_1,\ldots, X_n$ be indicator variables with $P(X_i = 1) = p_i$, $W = \sum_{i=1}^n X_i$, and $\lambda = EW = \sum_{i=1}^n p_i$. 
If for each $i \in [n]$, $(X_k^{(i)})_{k \neq i}$ has the distribution of $(X_k)_{k \neq i}$ conditional on $X_i = 1$ and $I$ is a random variable independent of all else such that $P(I = i) = p_i/\lambda$, 
then $W^s = \sum_{k \neq I} X_k^{(I)} + 1$ has the size-bias distribution of $W$. 
\end{proposition}

\subsection{Stein's method}

Stein's method is a powerful technique introduced by Charles Stein which is used to bound the distance between two probability distributions. 
It has been developed for many target distributions and successfully applied to establish central limit theorems in a wide range of settings. 
The main advantage of using Stein's method is that it provides an explicit error bound on the distributional approximation. We refer the reader to Ross \cite{Ros11} for an accessible survey of Stein's method. 

We use the following size-bias coupling version of Stein's method for Poisson approximation. 

\begin{theorem}[\cite{Ros11}, Theorem 4.13] \label{sizebiasStein}
Let $W \geq 0$ be an integer-valued random variable such that $EW = \lambda > 0$ and let $W^s$ be a size-bias coupling of $W$. If $Z \sim \Poi(\lambda)$, then 
\bal
d_{TV}(W,Z) \leq \min\{1,\lambda\} E|W + 1 - W^s|. 
\eal
\end{theorem}

\subsection{Dependency Graphs}

Let $\{Y_\alpha : \alpha \in A\}$ be a family of random variables. Then a graph $G$ is a {\em dependency graph} for the family $\{Y_\alpha : \alpha \in A\}$ if the following two conditions hold:
\begin{enumerate}
\item The vertex set of $G$ is $A$. 
\item If $A_1$ and $A_2$ are disconnected subsets in $G$, then $\{Y_\alpha : \alpha \in A_1\}$ and $\{Y_\alpha : \alpha \in A_2\}$ are independent.
\end{enumerate} 

Dependency graphs are useful for proving central limit theorems for sums of random variables without too many pairwise dependencies. 
Janson provided an asymptotic normality criterion in \cite{Jan88}, and Baldi and Rinott \cite{BR89} used Stein's method to obtain an upper bound on the Kolmogorov distance to the standard normal in terms of the dependency graph structure. 

We use the following stronger version of Stein's method via dependency graphs for normal approximation due to Hofer.  

\begin{theorem}[\cite{Hof18}, Theorem 3.5]\label{SteinDependencyGraph}
Let $Z \sim N(0,1)$. Let $G$ be a dependency graph for $\{Y_k\}_{k=1}^N$ and $D-1$ be the maximal degree of $G$. Let $X = \sum_{k=1}^N Y_k$. 
Assume there is a constant $B > 0$ such that $|Y_k - E(Y_k)| \leq B$ for all $k$. Then for $W = \frac{X - \mu}{\sigma}$, where $\mu = EX$ and $\sigma^2 = \Var(X)$, it holds that
\bal
d_K(W, Z) \leq \frac{8B^2D^{3/2}N^{1/2}}{\sigma^2} + \frac{8B^3D^2N}{\sigma^3} . 
\eal
\end{theorem}

\subsection{Concentration Inequalities}

The techniques developed for distributional approximation via Stein's method can also be used to prove concentration of measure inequalities (large deviations). These provide information on the tail behavior of distributions. 

The following theorem due to Ghosh and Goldstein gives concentration inequalities for the upper and lower tails, given a bounded size-bias coupling. 

\begin{theorem}[\cite{GG11}, Theorem 1.1]\label{sizebiasconcentration}
Let $Y$ be a nonnegative random variable with mean $\mu$ and variance $\sigma^2$, both finite and positive. 
Suppose there exists a coupling of $Y$ to a variable $Y^s$ having the $Y$-size-bias distribution which satisfies $|Y^s - Y| \leq C$ for some $C > 0$ with probability one. 
\begin{enumerate}
\item If $Y^s \geq Y$ with probability one, then
\bal
P\left( \frac{Y - \mu}{\sigma} \leq -t \right) \leq \exp\left( -\frac{t^2}{2A} \right)
\eal
for all $t > 0$, where $A = C\mu/\sigma^2$. 
\item If the moment generating function $m(\theta) = E\left(e^{\theta Y}\right)$ is finite at $\theta = 2/C$, then 
\bal
P\left( \frac{Y-\mu}{\sigma} \geq t \right) \leq \exp\left( -\frac{t^2}{2(A + Bt)} \right)
\eal
for all $t > 0$, where $A = C\mu/\sigma^2$ and $B = C/2\sigma$. 
\end{enumerate}
\end{theorem}

The next theorem established by Janson provides concentration inequalities for sums of dependent random variables with an appropriate dependency graph structure. 

\begin{theorem}[\cite{Jan04}, Theorem 2.1]\label{dependencyconcentration}
Suppose $X = \sum_{\alpha \in A} Y_\alpha$ with $a_\alpha \leq Y_\alpha \leq b_\alpha$ for all $\alpha \in A$ and some real numbers $a_\alpha$ and $b_\alpha$. 
Let $G$ be a dependency graph for $\{Y_\alpha : \alpha \in A\}$ and let $D-1$ be the maximal degree of $G$. Then for $t > 0$, 
\bal
P\left( X - EX \geq t \right) \leq \exp\left( -\frac{2t^2}{D\sum_{\alpha \in A} (b_\alpha - a_\alpha)^2} \right).
\eal
The same estimate holds for $P(X - EX \leq -t)$. 
\end{theorem}


\section{Fixed Points} \label{FixedPoints}

\subsection{Mean and Variance of Fixed Points}

Let $\sigma \in S_n$ be chosen uniformly at random from the double coset indexed by $T$. Let $\fp(\sigma)$ be the number of fixed points of $\sigma$, so that
\bal
\fp(\sigma) = \sum_{i=1}^n \theta_i,
\eal
where $\theta_i = \I_{\{\sigma(i) = i\}}$ is the indicator random variable that $\sigma$ has a fixed point at $i$.
 
We start by computing the probability that a value from $M_\ell$ occurs at an index in $L_k$. 

\begin{lemma}\label{fixedpointprob}
Let $\sigma \in S_n$ be a permutation chosen uniformly at random from the double coset in $S_\lambda \bslash S_n / S_\mu$ indexed by $T$. If $a \in L_k$ and $b \in M_\ell$, then 
\bal
P(\sigma(a) = b) = \frac{T_{k\ell}}{\lambda_k\mu_\ell}. \\
\eal
\end{lemma}

\begin{proof}
Suppose $a \in L_k$ and $b \in M_\ell$. Then
\bal
P(\sigma(a) = b) &= \frac{|\{\sigma : \sigma(a) = b, a \in L_k, b \in M_\ell\}|}{|S_\lambda \bslash \sigma / S_\mu|} \\
&= \frac{\displaystyle\prod_{i \neq k, j \neq \ell} \frac{\lambda_i! \mu_j!}{T_{ij}!} \cdot \frac{(\lambda_k-1)!(\mu_\ell - 1)!}{(T_{k\ell}-1)!}}{\displaystyle\prod_{i,j} \frac{\lambda_i! \mu_j!}{T_{ij}!}} \\
&= \frac{T_{k\ell}}{\lambda_k\mu_\ell}. \qedhere
\eal 
\end{proof}

Using the previous lemma, we can compute the expected number of fixed points.  

\begin{theorem}\label{fixedpointmean}
Let $\sigma \in S_n$ be a permutation chosen uniformly at random from the double coset in $S_\lambda \bslash S_n / S_\mu$ indexed by $T$. Then 
\bal
E(\fp(\sigma)) = \sum_{k,\ell} \frac{|A_{k\ell}|T_{k\ell}}{\lambda_k \mu_\ell}. 
\eal
\end{theorem}

\begin{proof}
By linearity of expectation and Lemma \ref{fixedpointprob}, 
\bal
E(\fp(\sigma)) &= \sum_{i=1}^n E\theta_i  = \sum_{k=1}^I \sum_{\ell=1}^J \sum_{i \in A_{k\ell}} \frac{T_{k\ell}}{\lambda_k \mu_\ell} = \sum_{k,\ell} \frac{|A_{k\ell}|T_{k\ell}}{\lambda_k \mu_\ell}. \qedhere
\eal
\end{proof}

Next we compute the variance of the number of fixed points. Let $\theta_{ij} = \I_{\{\sigma(i) = i, \sigma(j) = j\}}$ be the indicator random variable that $\sigma$ has fixed points at $i$ and $j$. 
We begin by computing the probability of having fixed points at two different indices. Note that it suffices to consider $i,j$ which are contained in nonempty $A_{k\ell}$. 

\begin{lemma}\label{fixedpointdoubleprob}
Let $\sigma \in S_n$ be a permutation chosen uniformly at random from the double coset in $S_\lambda \bslash S_n / S_\mu$ indexed by $T$. Suppose $i \in A_{k\ell}$ and $j \in A_{st}$. Then
\bal
P(\sigma(i) = i, \sigma(j) = j) = \begin{cases}
\frac{T_{k\ell} (T_{k\ell} - 1)}{\lambda_k (\lambda_k - 1) \mu_\ell (\mu_\ell - 1)} & \text{if $k = s$ and $\ell = t$} \\
\frac{T_{k\ell} T_{kt}}{\lambda_k (\lambda_k - 1) \mu_\ell \mu_t} & \text{if $k = s$ and $\ell \neq t$} \\
\frac{T_{k\ell} T_{s\ell}}{\lambda_k \lambda_s \mu_\ell (\mu_\ell - 1)} & \text{if $k \neq s$ and $\ell = t$} \\
\frac{T_{k\ell} T_{st}}{\lambda_k \lambda_s \mu_\ell \mu_t} & \text{if $k \neq s$ and $\ell \neq t$} \\
\end{cases}
\eal
\end{lemma}

\begin{proof}
This follows by similar computations as in the proof of Lemma \ref{fixedpointprob}. Suppose we are in the first case where $k = s$ and $\ell = t$, so that $i,j \in A_{k\ell}$. Then
\bal
P(\sigma(i) = i, \sigma(j) = j) &= \frac{|\{\sigma : \sigma(i) = i; \sigma(j) = j; i, j \in A_{k\ell}|}{|S_\lambda \bslash \sigma / S_\mu|} \\ 
&= \frac{\displaystyle\prod_{i \neq k, j \neq \ell} \frac{\lambda_i! \mu_j!}{T_{ij}!} \cdot \frac{(\lambda_k-2)!(\mu_\ell - 2)!}{(T_{k\ell}-2)!}}{\displaystyle\prod_{i,j} \frac{\lambda_i! \mu_j!}{T_{ij}!}} \\
&= \frac{T_{k\ell} (T_{k\ell} - 1)}{\lambda_k (\lambda_k - 1) \mu_\ell (\mu_\ell - 1)}.
\eal
The other three cases follow by similar computations and so we omit the details here. 
\end{proof}

Combining this lemma with the expected value obtained in Theorem \ref{fixedpointmean}, we can compute the variance.

\begin{theorem}\label{fixedpointvariance}
Let $\sigma \in S_n$ be a permutation chosen uniformly at random from the double coset in $S_\lambda \bslash S_n / S_\mu$ indexed by $T$. Then 
\bal
\Var(\fp(\sigma)) &= \sum_{k,\ell} \left( \frac{|A_{k\ell}|T_{k\ell}}{\lambda_k \mu_\ell}  - \frac{|A_{k\ell}|^2 T_{k\ell}^2}{\lambda_k^2 \mu_\ell^2} + \frac{|A_{k\ell}|(|A_{k\ell}| - 1) T_{k\ell} (T_{k\ell} - 1)}{\lambda_k (\lambda_k - 1) \mu_\ell (\mu_\ell - 1)} \right) \\
& + 2 \sum_{\substack{k \\ \ell < t}} \frac{|A_{k\ell}||A_{kt}|T_{k\ell} T_{kt}}{\lambda_k (\lambda_k - 1) \mu_\ell \mu_t} + 2 \sum_{\substack{\ell \\ k < s}} \frac{|A_{k\ell}| |A_{s\ell}| T_{k\ell} T_{s\ell}}{\lambda_k \lambda_s \mu_\ell (\mu_\ell - 1)}. 
\eal
\end{theorem}

\begin{proof}
By Theorem \ref{fixedpointmean} and Lemma \ref{fixedpointdoubleprob} we have
\bal
E(\fp(\sigma)^2) &= \sum_{i=1}^n E(\theta_i) + 2 \sum_{i < j} E(\theta_{ij}) \\
&= \sum_{k,\ell} \frac{|A_{k\ell}|T_{k\ell}}{\lambda_k \mu_\ell} + 2\sum_{k,\ell} \sum_{i,j \in A_{k\ell}} \frac{T_{k\ell} (T_{k\ell} - 1)}{\lambda_k (\lambda_k - 1) \mu_\ell (\mu_\ell - 1)} \\
& + 2\sum_{\substack{k \\ \ell < t}} \sum_{\substack{i \in A_{k\ell} \\ j \in A_{kt}}} \frac{T_{k\ell} T_{kt}}{\lambda_k (\lambda_k - 1) \mu_\ell \mu_t} 
+ 2\sum_{\substack{\ell \\ k < s}} \sum_{\substack{i \in A_{k\ell} \\ j \in A_{s\ell}}} \frac{T_{k\ell} T_{s\ell}}{\lambda_k \lambda_s \mu_\ell (\mu_\ell - 1)} \\
& + 2\sum_{\substack{k < s \\ \ell < t}} \sum_{\substack{i \in A_{k\ell} \\ j \in A_{st}}} \frac{T_{k\ell} T_{st}}{\lambda_k \lambda_s \mu_\ell \mu_t} 
\eal

Moreover, 
\bal
E(\fp(\sigma))^2 =  \sum_{k,\ell} \frac{|A_{k\ell}|^2 T_{k\ell}^2}{\lambda_k^2 \mu_\ell^2} + 2 \sum_{\substack{k < s \\ \ell < t}} \frac{|A_{k\ell}| |A_{st}| T_{k\ell} T_{st}}{\lambda_k \lambda_s \mu_\ell \mu_t}
\eal

Therefore, using $\Var(\fp(\sigma)) = E(\fp(\sigma)^2) - E(\fp(\sigma))^2$ and simplifying yields the desired variance. 
\end{proof}

\subsection{Two-Cycles} 

In order to apply the size-bias coupling version of Stein's method, we will need the probability that $\sigma$ has a two-cycle at some given indices. 

\begin{lemma}\label{twocycleprob}
Let $\sigma \in S_n$ be a permutation chosen uniformly at random from the double coset in $S_\lambda \bslash S_n / S_\mu$ indexed by $T$. Suppose $a \in M_s$, $a \in L_\ell$, $b \in M_t$, and $b \in L_r$. 
Then the probability that $\sigma$ has a two-cycle at positions $a$ and $b$ is
\bal
P(\sigma(a) = b, \sigma(b) = a) = 
\begin{cases} 
\frac{T_{\ell s}T_{r t}}{\lambda_\ell \lambda_r \mu_s \mu_t} & \text{if } \ell \neq r \text{ and } s \neq t \\
\frac{T_{\ell s}T_{\ell t}}{\lambda_\ell (\lambda_\ell - 1) \mu_s \mu_t} & \text{if } \ell = r \text{ and } s \neq t \\
\frac{T_{\ell s}T_{r s}}{\lambda_\ell \lambda_r \mu_s (\mu_s - 1)} & \text{if } \ell \neq r \text{ and } s = t \\
\frac{T_{\ell s}(T_{\ell s} - 1)}{\lambda_\ell (\lambda_\ell - 1) \mu_s (\mu_s - 1)} & \text{if } \ell = r \text{ and } s = t 
\end{cases}
\eal
\end{lemma}

\begin{proof}
Suppose we are in the first case where $\ell \neq r$ and $s \neq t$. Then
\bal
P(\sigma(a) = b, \sigma(b) = a) &= \frac{|\{\sigma : \sigma(a) = b, \sigma(b) = a, a \in M_s, a \in L_\ell, b \in M_t, b \in L_r\}|}{|S_\lambda \bslash \sigma / S_\mu|} \\
&= \frac{\displaystyle\prod_{\substack{i \neq j \neq \ell \\ i \neq j \neq r}} \frac{\lambda_i! \mu_j!}{T_{ij}!} \cdot \frac{(\lambda_\ell - 1)!(\mu_s - 1)!}{(T_{\ell s} - 1)!} 
\cdot \frac{(\lambda_r - 1)!(\mu_t - 1)!}{(T_{r t} - 1)!}}{\displaystyle\prod_{i,j} \frac{\lambda_i! \mu_j!}{T_{ij}!}} \\
&= \frac{T_{\ell s}T_{r t}}{\lambda_\ell \lambda_r \mu_s \mu_t}.
\eal 
The other three cases follow by similar computations and so we omit the details here. 
\end{proof}


\subsection{Poisson Limit Theorem for Fixed Points} \label{PoissonApproximationForFixedPoints} 

We now prove the Poisson limit theorem for the number of fixed points. 

\begin{proof}[Proof of Theorem \ref{PoissonLimitTheorem}] 
Let $W := \fp(\sigma) = \sum_{i=1}^n \theta_i$, where $\theta_i = \I_{\{\sigma(i) = i\}}$. By Proposition \ref{fixedpointmean}, the expected number of fixed points is
\bal
EW = \sum_{k,\ell} \frac{|A_{k\ell}|T_{k\ell}}{\lambda_k \mu_\ell} =: \nu_n.
\eal

We start by defining a size-bias distribution, $W^s$, of $W$ as follows. Pick $I$ independent of all else with probability $P(I = i) = \frac{E\theta_i}{EW} = \frac{T_{k\ell}}{\nu_n \lambda_k \mu_\ell}$, for $i \in L_k$ and $i \in M_\ell$. 
Suppose that $\sigma(X) = I$. If $X = I$, then $I$ is already a fixed point of $\sigma$, so we simply set $\sigma^s = \sigma$. 

So suppose $X \neq I$ so that $\sigma$ does not have a fixed point at $I$. If the indices $I$ and $X$ are in the same $\lambda$-block or if the values $I$ and $\sigma(I)$ are in the same $\mu$-block, 
then we simply swap the values $\sigma(I)$ and $I$ to get a new permutation with a fixed point at index $I$. 
Set $\sigma^s$ to be the resulting permutation. Observe that this swap does not change the value of $T_{ij}$ for all $i,j$. 

Otherwise, indices $I$ and $X$ are in different $\lambda$-blocks and values $\sigma(I)$ and $I$ are in different $\mu$-blocks. 
Suppose $I \in L_k$, $\sigma(I) \in M_a$, and $X \in L_r$. 
Pick an element $z$ uniformly at random from the set $\{\sigma(y) \in M_\ell : y \in L_k \}$, which is of size $T_{k\ell}$, and then swap $\sigma(I)$ and $z$. 
Note that we can always find such a $z$ since $T_{k\ell} \geq 1$ by hypothesis. 
Finally, swap the values $z$ and $\sigma(X)$, so that the resulting permutation has a fixed point at $I$. Set $\sigma^s$ to be the resulting permutation.
Note that in this case, it is not enough to simply swap the values $\sigma(I)$ and $I$ since we must account for the change in the values of $T_{ij}$ when we are making swaps. 
Observe that this two step swap also does not change the value of $T_{ij}$ for all $i,j$. 

By construction, $\sigma^s$ obtained by the coupling above has the same marginal distribution as a random permutation from the double coset indexed by $T$, conditioned to have a fixed point at position $I$. 
Indeed the distribution of a random permutation from the double coset indexed by $T$ only depends on the counts $T_{ij}$ and not on the positions of values within blocks. 

Set $W^s = \fp(\sigma^s) = \sum_{\beta \neq I} \theta_\beta^{(I)} + 1$, where $\theta_\beta^{(I)}$ is the indicator random variable that $\sigma^s$ has a fixed point at position $\beta$.
By Proposition \ref{sizebiascouplingconstruction}, $W^s$ has the size-bias distribution of $W$. Thus $(W, W^s)$ is a size-bias coupling. 

Let $V_i := \sum_{\beta \neq i} \theta_\beta^{(i)}$, so that $W^s = V_I + 1$. Let $Y_n \sim \Poi(\nu_n)$. By Theorem \ref{sizebiasStein}, 
\bal
d_{TV}(W,Y_n) &\leq \min\{1, \nu_n\} E|W + 1 - W^s| \\ 
&= \min\{1, \nu_n\} \sum_{i=1}^n P(I = i) E|W - V_i| \\
&= \min\{1, \nu_n^{-1}\} \sum_{k,\ell} \sum_{i \in A_{k\ell}} \frac{T_{k\ell}}{\lambda_k \mu_\ell} E|W - V_i| \\
&= \min\{1, \nu_n^{-1}\} \sum_{k,\ell} \frac{|A_{k\ell}| T_{k\ell}}{\lambda_k \mu_\ell} E\left|W - V_i\right|.  \\
\eal

Next, we compute the random variable $\left|W - V_i\right| = \left|\theta_i + \sum_{\beta \neq i} \left(\theta_\beta - \theta_\beta^{(i)}\right)\right|$. 
Suppose $i \in L_k$ and $i \in M_\ell$. There are several cases.

\textbf{Case 1}. If $i$ is a fixed point, then $|W - V_i| = 1$. By Lemma \ref{fixedpointprob}, this occurs with probability $\frac{T_{k\ell}}{\lambda_k \mu_\ell}$. 

So we may assume that $i$ is not a fixed point. Let $\sigma(x) = i$. Suppose $\sigma(i) \in M_a$ and $x \in L_r$. 

\textbf{Case 2}. If $\ell = a$ so that values $i, \sigma(i)$ are in the same $\mu$-block, or if $k = r$ so that indices $i, x$ are in the same $\lambda$-block, then the coupling swaps the values $i$ and $\sigma(i)$. 
In this case, $|W - V_i| = \I_{\{i, \sigma(i) \text{ are in a two-cycle}\}} = \I_{\{\sigma(i) = x, \sigma(x) = i\}}$. Thus by Lemma \ref{twocycleprob}, the probability that $i$ and $x$ are in a two-cycle is 
\bal
P(\text{$i$ and $x$ are in a two-cycle}) &= \frac{T_{k\ell} T_{r\ell}}{\lambda_k \lambda_r \mu_\ell(\mu_\ell - 1)}\I_{\{x \in M_\ell, x \in L_r, \ell = a, k \neq r\}} \\ 
&\qquad + \frac{T_{k\ell} T_{k a}}{\lambda_k (\lambda_k - 1) \mu_\ell \mu_a} \I_{\{x \in M_a, x \in L_k, \ell \neq a, k = r\}} \\ 
& \qquad + \frac{T_{k\ell} (T_{k\ell} - 1)}{\lambda_k (\lambda_k - 1) \mu_\ell (\mu_\ell - 1)} \I_{\{x \in M_\ell, x \in L_k, \ell = a, k = r\}}. 
\eal

\textbf{Case 3}. Otherwise $\ell \neq a$ and $k \neq r$ so that values $i, \sigma(i)$ are in different $\mu$-blocks and indices $i, x$ are in different $\lambda$-blocks.  
Then the coupling picks $z \in \{\sigma(y) \in M_\ell : y \in L_k \}$ uniformly at random, swaps values $z$ and $\sigma(i)$, and then swaps values $z$ and $i$. 
By construction, $|W - V_i| = \I_{\{z \text{ is a fixed point}\}}$. Again by Lemma \ref{fixedpointprob}, this occurs with probability $\frac{T_{k\ell}}{\lambda_k \mu_\ell}$. 

Combining the cases above gives the expected value 
\bal
E|W - V_i| &= P(i \text{ is a fixed point}) + P(i \text{ is in a two-cycle}) \\
&\qquad + P(z \text{ is a fixed point}) \\
&= \frac{T_{k\ell}}{\lambda_k \mu_\ell} + \sum_{x \neq i} P(\text{$i$ and $x$ are in a two-cycle}) + \frac{T_{k\ell}}{\lambda_k \mu_\ell} \\
&\leq \frac{5T_{k\ell}}{\lambda_k \mu_\ell}.
\eal 

Thus the total variation distance is upper bounded by
\bal
d_{TV}(W,Y_n) &\leq \min\{1, \nu_n^{-1}\} \sum_{k,\ell} \frac{|A_{k\ell}| T_{k\ell}}{\lambda_k \mu_\ell} \left( \frac{5T_{k\ell}}{\lambda_k \mu_\ell} \right) \\
&\leq 5 \min\{1, \nu_n^{-1}\} \sum_{k,\ell} \frac{|A_{k\ell}|}{\lambda_k \mu_\ell} \\
&\leq 5 \min\{1, \nu_n^{-1}\} \sum_{k,\ell} \frac{\min\{\lambda_k, \mu_\ell\}}{\lambda_k\mu_\ell} \\
&\leq 5 \min\{1, \nu_n^{-1}\} \sum_{k,\ell} \frac{1}{\max\{\lambda_k, \mu_\ell\}} \\
&\leq \frac{5 (I+J-1) \min\{1, \nu_n^{-1}\}}{\max\{\lambda_I, \mu_J\}}
\eal
where we used the facts that $|A_{k\ell}| \leq \min\{\lambda_k, \mu_\ell\}$ for all $k,\ell$, the size of the partition $\{A_{k\ell}\}$ is at most $I+J-1$, and $\lambda_k \geq \lambda_I$ for all $k$ and $\mu_\ell \geq \mu_J$ for all $\ell$. 

Let $Y \sim \Poi(\nu)$. By hypothesis, since $\limsup_{n \to \infty} \nu_n = \nu$, $I + J - 1 \to C$, $\lambda_I \to \infty$, and $\mu_J \to \infty$, we have that $d_{TV}(W,Y_n) \to 0$ and $d_{TV}(Y_n,Y) \to 0$ as $n \to \infty$. 
Therefore $d_{TV}(W, Y) \to 0$ as $n \to \infty$, and thus $W \xrightarrow{d} Y$. 
\end{proof} 

As a corollary, we can recover Montmort's classical result on the Poisson limit theorem for the number of fixed points in a uniformly random permutation, along with a rate of convergence. 

\begin{corollary}[\cite{Mon08}]
Let $\sigma \in S_n$ be a permutation chosen uniformly at random. Then 
\bal
d_{TV}(\fp(\sigma), Y) \leq \frac{5}{n},
\eal
where $Y \sim \Poi(1)$. Thus $\fp(\sigma) \xrightarrow{d} \Poi(1)$ as $n \to \infty$. 
\end{corollary}

\begin{proof}
Setting $\lambda = \mu = (n)$ in Theorem \ref{PoissonLimitTheorem} gives $d_{TV}(\fp(\sigma),Y) \leq \frac{5}{n}$, where $Y \sim \Poi(1)$. Therefore, $d_{TV}(W, Y) \to 0$ as $n \to \infty$, and so $\fp(\sigma) \xrightarrow{d} \Poi(1)$. 
\end{proof}

In \cite{CD18} Crane and DeSalvo used Stein's method and size-bias coupling to obtain a total variation upper bound of $\frac{3}{n}$ for uniformly random permutations, which is slightly better than our bound. 
However, using properties of alternating series and the definition of total variation distance directly, Arratia and Tavare \cite{AT92} showed a super-exponential upper bound of $\frac{2^n}{(n+1)!}$ for uniformly random permutations.


\section{Descents} \label{Descents} 

Let $\sigma \in S_n$ be chosen uniformly at random from the double coset in $S_\lambda \bslash S_n / S_\mu$ indexed by $T$. Let $\des(\sigma)$ be the number of descents of $\sigma$, so that 
\bal
\des(\sigma) &= \sum_{i=1}^{n-1} I_i, 
\eal
where $I_i = \I_{\{\sigma(i) > \sigma(i+1)\}}$ is the indicator random variable that $\sigma$ has a descent at index $i$. 

\subsection{Mean and Variance of Descents} \label{MeanVarDescents}

In this section we compute the mean and asymptotic variance of $\des(\sigma)$. We start by computing the probability that a descent occurs at some fixed position. 
In what follows, let $\lambda = (\lambda_1,\ldots, \lambda_I)$ and $\mu = (\mu_1, \ldots, \mu_J)$ be two partitions of $n$. 

\begin{lemma}\label{descentprob}
Let $\sigma$ be a permutation chosen uniformly at random from the double coset in $S_\lambda \bslash S_n / S_\mu$ indexed by $T$. Then
\bal
P(I_i = 1) = \begin{cases} 
\frac{1}{2} & \text{if $i \in L_k^o$, for $k \in [I]$},  \\ 
\frac{1}{2} - \frac{f(k)}{2\lambda_k\lambda_{k+1}} & \text{if $i = l_k^b$, for $k \in [I-1]$},
\end{cases}
\eal
where $f(k) := \sum_{1 \leq a < b \leq J} (T_{ka}T_{k+1,b} - T_{k+1,a}T_{kb})$, for $k \in [I-1]$. 
\end{lemma}

\begin{proof}
Suppose that $i \in L_k$ for $k \in [I]$. There are two cases for the position of $i$:
\begin{itemize}
\item {\bf Case 1 (Interior):} $i \in L_k^o$, for $k \in [I]$,
\item {\bf Case 2 (Border):} $i = l_k^b$, for $k \in [I-1]$.
\end{itemize}

First suppose $i \in L_k^o$. 
Then swapping the values $\sigma(i)$ and $\sigma(i+1)$ maps a uniformly random permutation in $S_\lambda \bslash S_n / S_\mu$ to a permutation with the same distribution, 
and transforms a permutation with a descent at $i$ to a permutation without a descent at $i$. It follows that $P(I_i = 1, \text{Case 1}) = \frac{1}{2}$. 

So suppose $i = l_k^b$. There are two subcases for the values of $\sigma(i)$ and $\sigma(i+1)$:
\begin{itemize}
\item {\bf Subcase 1 (Same block):} $\sigma(i), \sigma(i+1) \in M_\ell$, for $\ell \in [J]$, 
\item {\bf Subcase 2 (Different blocks):} $\sigma(i) \in L_b$ and $\sigma(i+1) \in L_a$, for $1 \leq a < b \leq J$. 
\end{itemize}
Then 
\bal
P(I_i = 1, \text{Case 2}, \text{Subcase 1}) &= \frac{\frac{\binom{\mu_\ell}{2} (\mu_\ell - 2)!(\lambda_k - 1)!(\lambda_{k+1} - 1)!}{(T_{k\ell} - 1)!(T_{k+1,\ell} - 1)!} 
\frac{\displaystyle\prod_{i \notin \{k,k+1\}, j \neq \ell} \lambda_i! \mu_j!}{\displaystyle\prod_{(i,j) \notin \{(k,\ell), (k+1,\ell)\}} T_{ij}!}}{\displaystyle\prod_{i,j} \frac{\lambda_i! \mu_j!}{T_{ij}!}}  \\
&= \frac{T_{k\ell}T_{k+1,\ell}}{2\lambda_k \lambda_{k+1}}, \\
P(I_i = 1, \text{Case 2}, \text{Subcase 2}) &= \frac{\frac{ \mu_a (\mu_a - 1)! \mu_b (\mu_b - 1)! (\lambda_k - 1)!(\lambda_{k+1} - 1)!}{(T_{kb} - 1)!(T_{k+1,a} - 1)!}
\frac{\displaystyle\prod_{i \notin \{k,k+1\}, j \notin \{a,b\}} \lambda_i! \mu_j!}{\displaystyle\prod_{(i,j) \notin \{(k+1,b), (k,a)\}}  T_{ij}!}}{\displaystyle\prod_{i,j} \frac{\lambda_i! \mu_j!}{T_{ij}!}}  \\
&= \frac{T_{k+1,a}T_{kb}}{\lambda_k \lambda_{k+1}}.
\eal
Summing over the two subcases yields
\bal
&P(I_i = 1, \text{Case $2$}) \\
&= \sum_{\ell = 1}^J P(I_i = 1, \text{Case 2}, \text{Subcase 1}) + \sum_{1 \leq a < b \leq J} P(I_i = 1, \text{Case 2}, \text{Subcase 2}) \\
& = \frac{1}{2\lambda_k\lambda_{k+1}} \left(\sum_{\ell = 1}^J T_{k\ell}T_{k+1,\ell} + 2 \sum_{1 \leq a < b \leq J} T_{k+1,a}T_{kb}\right)  \\
& = \frac{1}{2\lambda_k\lambda_{k+1}} \left( \left(\sum_{\ell = 1}^J T_{k\ell}\right)\left(\sum_{\ell = 1}^J T_{k+1,\ell}\right) - \sum_{1 \leq a < b \leq J} ( T_{ka}T_{k+1,b} - T_{k+1,a}T_{kb} )  \right) \\
& = \frac{\lambda_k\lambda_{k+1}}{2\lambda_k\lambda_{k+1}} - \frac{1}{2\lambda_k\lambda_{k+1}} \sum_{1 \leq a < b \leq J}  (T_{ka}T_{k+1,b} - T_{k+1,a}T_{kb}) \\
& = \frac{1}{2} - \frac{f(k)}{2\lambda_k\lambda_{k+1}}. \qedhere
\eal
\end{proof}

Next we compute the probability that a descent occurs in two consecutive positions. 
Let $I_{ij} = I_iI_j$ denote the indicator random variable for the event that there are descents at both positions $i$ and $j$. 
Define the set $L_k^{oo} := L_k \setminus \{l_k^b, l_k^b - 1\}$, for $1 \leq k \leq I$. 

\begin{lemma}\label{descentprobtwoconsecutive} 
Let $\sigma$ be a permutation chosen uniformly at random from the double coset in $S_\lambda \bslash S_n / S_\mu$ indexed by $T$. Then
\bal
P(I_{i,i+1} = 1) = 
\begin{cases}
\frac{1}{6} & \text{$i \in L_k^{oo}$, for $k \in [I]$} \\ 
\frac{g(k)}{6\lambda_k(\lambda_k - 1)\lambda_{k+1}} & \text{if $i = l_k^b - 1$, for $k \in [I-1]$} \\ 
\frac{h(k)}{6\lambda_k\lambda_{k+1}(\lambda_{k+1} - 1)} & \text{if $i = l_k^b$, for $k \in [I-1]$},   
\end{cases}
\eal
where 
\bal
g(k) &:= \sum_{\ell=1}^J T_{k\ell}(T_{k\ell}-1)T_{k+1,\ell} \\ &\qquad + 3\sum_{1 \leq a < b \leq J} (T_{kb}(T_{kb}-1)T_{k+1,a} + T_{kb}T_{ka}T_{k+1,a}) \\ &\qquad + 6\sum_{1 \leq a < b < c \leq J} T_{kc}T_{kb}T_{k+1,a}, \\
h(k) &:=  \sum_{\ell=1}^J T_{k\ell}T_{k+1,\ell}(T_{k+1,\ell}-1) \\ &\qquad + 3\sum_{1 \leq a < b \leq J} (T_{kb}T_{k+1,b}T_{k+1,a} + T_{kb}T_{k+1,a}(T_{k+1,a}-1)) \\ &\qquad + 6\sum_{1 \leq a < b < c \leq J} T_{kc}T_{k+1,b}T_{k+1,a}.
\eal
\end{lemma}

\begin{proof}
Suppose that $i \in L_k$ for $k \in [I]$. There are three cases for the position of $i$:
\begin{itemize}
\item {\bf Case 1 (Interior):} $i \in L_k^{oo}$, for $k \in [I]$,
\item {\bf Case 2 (Sub-Border):} $i = l_k^b - 1$, for $k \in [I-1]$.
\item {\bf Case 3 (Border):} $i = l_k^b$, for $k \in [I-1]$.
\end{itemize}

Suppose $i \in L_k^{oo}$. It is straightforward that $P(I_{i,i+1} = 1) = \frac{1}{6}$ by a similar argument as in the proof of Lemma \ref{descentprob}. 

So suppose we are in Cases 2 and 3. For each case, there are four subcases for the values of $\sigma(i), \sigma(i+1), \sigma(i+2)$:
\begin{itemize}
\item {\bf Subcase 1:} $\sigma(i), \sigma(i+1), \sigma(i+2) \in M_\ell$, for $\ell \in [J]$
\item {\bf Subcase 2:} $\sigma(i), \sigma(i+1) \in M_b, \sigma(i+2) \in M_a$, for $1 \leq a < b \leq J$
\item {\bf Subcase 3:} $\sigma(i) \in M_b, \sigma(i+1), \sigma(i+2) \in M_a$, for $1 \leq a < b \leq J$
\item {\bf Subcase 4:} $\sigma(i) \in M_c, \sigma(i+1) \in M_b, \sigma(i+2) \in M_a$, for $1 \leq a < b < c \leq J$
\end{itemize}
Define $p_{xy} := P(I_{i, i+1} = 1, \text{ Case $x$}, \text{ Subcase $y$})$. 

We first consider Case $2$. By similar computations as in the proof of Lemma \ref{descentprob}, we get
\bal
p_{21} &= \frac{T_{k\ell}(T_{k\ell}-1)T_{k+1,\ell}}{6\lambda_k(\lambda_k-1)\lambda_{k+1}}, \qquad p_{22} = \frac{T_{kb}(T_{kb}-1)T_{k+1,a}}{2\lambda_k(\lambda_k-1)\lambda_{k+1}}, \\
p_{23} &= \frac{T_{kb}T_{ka}T_{k+1,a}}{2\lambda_k(\lambda_k-1)\lambda_{k+1}}, \qquad p_{24} = \frac{T_{kc}T_{kb}T_{k+1,a}}{\lambda_k(\lambda_k-1)\lambda_{k+1}}. \\
\eal
This gives
\bal
P(I_{i,i+1} = 1, \text{Case $2$}) &= \sum_{\ell = 1}^J p_{21} + \sum_{1 \leq a < b \leq J} (p_{22} + p_{23}) + \sum_{1 \leq a < b < c \leq J} p_{24} \\
&= \frac{g(k)}{6\lambda_k(\lambda_k - 1)\lambda_{k+1}}.
\eal

Similarly for Case $3$ we get
\bal
p_{31} &= \frac{T_{k\ell}T_{k+1,\ell}(T_{k+1,\ell}-1)}{6\lambda_k\lambda_{k+1}(\lambda_{k+1}-1)}, \qquad p_{32} = \frac{T_{kb}T_{k+1,b}T_{k+1,a}}{2\lambda_k\lambda_{k+1}(\lambda_{k+1}-1)}, \\
p_{33} &= \frac{T_{kb}T_{k+1,a}(T_{k+1,a}-1)}{2\lambda_k\lambda_{k+1}(\lambda_{k+1}-1)}, \qquad p_{34} = \frac{T_{kc}T_{k+1,b}T_{k+1,a}}{\lambda_k\lambda_{k+1}(\lambda_{k+1}-1)}, 
\eal
so that
\bal
P(I_{i,i+1} = 1, \text{Case $3$}) &= \sum_{\ell = 1}^J p_{31} + \sum_{1 \leq a < b \leq J} (p_{32} + p_{33}) + \sum_{1 \leq a < b < c \leq J} p_{34} \\
&= \frac{h(k)}{6\lambda_k\lambda_{k+1}(\lambda_{k+1} - 1)}. \qedhere
\eal
\end{proof}

The final lemma we need is the probability that two descents occur at positions $i$ and $j$, for $|j-i| \geq 2$. 
Note that this lemma shows that descents at indices $i$ and $j$ are independent except for the cases where $i$ and $j$ are in consecutive border indices. 

\begin{lemma}\label{doubledescentprob} 
Let $\sigma$ be a permutation chosen uniformly at random from the double coset in $S_\lambda \bslash S_n / S_\mu$ indexed by $T$. Then for $|j-i| \geq 2$, 
\bal
&P(I_{ij} = 1) = \\
&\begin{cases}
\frac{1}{4}, & \text{if $i,j \in L_k^{oo}$, for $1 \leq k \leq I$,} \\ & \text{or if $i \in L_\ell^o$ and $j \in L_r^o$, for $1 \leq \ell < r \leq I$;} \\
\frac{1}{4} - \frac{f(k)}{4\lambda_k\lambda_{k+1}}, & \text{if exactly one of $i$ or $j$ equals $l_k^b$, for $k \in [I-1]$;} \\
\left(\frac{1}{2} - \frac{f(\ell)}{2\lambda_\ell\lambda_{\ell+1}}\right)\left(\frac{1}{2} - \frac{f(r)}{2\lambda_r\lambda_{r+1}}\right), & \text{if $i = l_\ell^b$ and $j = l_r^b$, for $1 \leq \ell < r \leq I-1$}, \\ & \text{and $|r - \ell| \geq 2$}.
\end{cases}
\eal
If $i = l_k^b$ and $j = l_{k+1}^b$ for $k \in [I-2]$, then
\bal
&P(I_{ij} = 1) \\
&= \frac{\lambda_k}{4(\lambda_k - 1)} - \frac{f(k+1)}{4(\lambda_{k+1}-1)\lambda_{k+2}} - \frac{f(k)}{4\lambda_k(\lambda_{k+1}-1)} + \frac{f(k)f(k+1) - e(k)}{4\lambda_k \lambda_{k+1} (\lambda_{k+1}-1) \lambda_{k+2}},
\eal
where $f(k)$ is defined as in Lemma \ref{descentprob} and
\bal
e(k) &:= \sum_{\ell=1}^J T_{k\ell}T_{k+1,\ell}T_{k+2,\ell} \\
&\qquad + 2 \left( \sum_{1 \leq a < b \leq J} T_{kb}T_{k+1,a} T_{k+2,a} - \sum_{1 \leq c < d \leq J} T_{kd}T_{k+1,d} T_{k+2,c} \right) \\
&\qquad + 4\sum_{1 \leq c < a < b \leq J} T_{kb}T_{k+1,a} T_{k+2,c}.
\eal
\end{lemma}

\begin{proof}
There are three cases for the positions of $i$ and $j$. If we are in the first case so that $i,j \in L_k^{oo}$ for some $1 \leq k \leq I$ or if $i \in L_\ell^o$ and $j \in L_r^o$ for $1 \leq \ell < r \leq I$, 
then it is clear that $P(I_{ij} = 1) = \frac{1}{4}$ by similar arguments as in the proofs of Lemmas \ref{descentprob} and \ref{descentprobtwoconsecutive}. 

So suppose that we are in the second or third case. Then 
\bal
P(I_{ij} = 1) = P(I_i = 1, I_j = 1) = P(I_i = 1 \mid I_j = 1)P(I_j = 1). 
\eal

Suppose we are in the second case, so that exactly one of $i$ or $j$ equals $l_k^b$ for some $k \in [N-1]$. Without loss of generality assume $j = l_k^b$, so that $i \in L_\ell^o$ for some $\ell \neq k$ or $i \in L_k^{oo}$. 
By Lemma \ref{descentprob}, $P(I_j = 1) = \frac{1}{2} - \frac{f(k)}{2\lambda_k \lambda_{k+1}}$. 
Also by Lemma \ref{descentprob}, $P(I_i = 1 \mid I_j = 1) = \frac{1}{2}$, since the event of a descent at in interior index is independent of descents in other positions. 
This case gives $P(I_{ij} = 1) = \frac{1}{2}\left( \frac{1}{2} - \frac{f(k)}{2\lambda_k \lambda_{k+1}} \right) = \frac{1}{4} - \frac{f(k)}{4\lambda_k\lambda_{k+1}}$.

Finally suppose we are in the third case so that $i = l_\ell^b$ and $j = l_r^b$, for $1 \leq \ell < r \leq I-1$. 
By Lemma \ref{descentprob}, the probability of a descent at the border index $l_\ell^b$ only depends on the values of $\{T_{\ell, m}\}_{1 \leq m \leq J}$, $\{T_{\ell+1, m}\}_{1 \leq m \leq J}$, $\lambda_\ell$, and $\lambda_{\ell + 1}$. 
It follows that if $|r - \ell| \geq 2$, then $P(I_i = 1 \mid I_j = 1) = P(I_i = 1)$, so that $P(I_{ij} = 1) = \left(\frac{1}{2} - \frac{f(\ell)}{2\lambda_\ell\lambda_{\ell+1}}\right)\left(\frac{1}{2} - \frac{f(r)}{2\lambda_r\lambda_{r+1}}\right)$.

Otherwise we have that $i = l_k^b$ and $j = l_{k+1}^b$ for some $k \in [I-2]$. In this case indices $i+1$ and $j$ are both in $L_{k+1}$. Suppose $\sigma(i) \in M_b$, $\sigma(i+1) \in M_a$, $\sigma(j) \in M_d$, and $\sigma(j+1) \in M_c$, for $a,b,c,d \in [J]$. 
There are eight subcases:
\begin{itemize}
\item {\bf Subcase 1:} $a = b = \ell, c = d = r, \ell \neq r$
\item {\bf Subcase 2:} $a = b = \ell, c = d = r, \ell = r$
\item {\bf Subcase 3:} $a < b, c = d = r, a \neq r$
\item {\bf Subcase 4:} $a < b, c = d = r, a = r$
\item {\bf Subcase 5:} $a = b = \ell, c < d, \ell \neq d$
\item {\bf Subcase 6:} $a = b = \ell, c < d, \ell = d$.
\item {\bf Subcase 7:} $a < b, c < d, a \neq d$
\item {\bf Subcase 8:} $a < b, c < d, a = d$.
\end{itemize}
Let $p_{x} := P(I_{ij} = 1, \text{ Subcase x})$. Then by computations similar to Lemmas \ref{descentprob} and \ref{descentprobtwoconsecutive} we get
\bal
p_1 &= \frac{T_{k\ell}T_{k+1,\ell}T_{k+1,r}T_{k+2,r}}{4\lambda_k \lambda_{k+1} (\lambda_{k+1}-1) \lambda_{k+2}}, \qquad p_2 = \frac{T_{k\ell}T_{k+1,\ell}(T_{k+1,\ell} - 1)T_{k+2,\ell}}{4 \lambda_k \lambda_{k+1} (\lambda_{k+1}-1) \lambda_{k+2}} \\
p_3 &= \frac{T_{kb}T_{k+1,a} T_{k+1,r} T_{k+2,r}}{2\lambda_k \lambda_{k+1} (\lambda_{k+1}-1) \lambda_{k+2}} \qquad p_4 = \frac{T_{kb}T_{k+1,r} (T_{k+1,r}-1) T_{k+2,r}}{2 \lambda_k \lambda_{k+1} (\lambda_{k+1}-1) \lambda_{k+2}}  \\
p_5 &= \frac{T_{k\ell}T_{k+1,\ell} T_{k+1,d} T_{k+2,c}}{2 \lambda_k \lambda_{k+1} (\lambda_{k+1}-1) \lambda_{k+2}} \qquad p_6 =  \frac{T_{k\ell}T_{k+1,\ell} (T_{k+1,\ell}-1) T_{k+2,c}}{2 \lambda_k \lambda_{k+1} (\lambda_{k+1}-1) \lambda_{k+2}}\\
p_7 &= \frac{T_{kb}T_{k+1,a} T_{k+1,d} T_{k+2,c}}{\lambda_k \lambda_{k+1} (\lambda_{k+1}-1) \lambda_{k+2}}, \qquad p_8 =  \frac{T_{kb}T_{k+1,a} (T_{k+1,a}-1) T_{k+2,c}}{\lambda_k \lambda_{k+1} (\lambda_{k+1}-1) \lambda_{k+2}} 
\eal
Summing over the subcases and simplifying gives
\bal
&P(I_{ij} = 1)  \\
&= \frac{1}{4\lambda_k \lambda_{k+1}(\lambda_{k+1}-1)\lambda_{k+2}} \left[ \left(\sum_{\ell = 1}^J T_{k\ell}T_{k+1,\ell} + 2 \sum_{1 \leq a < b \leq J} T_{kb}T_{k+1,a}\right) \right. \\
&\qquad \times \left. \left(\sum_{r = 1}^J T_{k+1,r}T_{k+2,r} + 2 \sum_{1 \leq c < d \leq J} T_{k+1,d}T_{k+2,c}\right) \right. \\
&\qquad - \left. \sum_{\ell=1}^J T_{k\ell}T_{k+1,\ell}T_{k+2,\ell} - 2 \left( \sum_{1 \leq a < b \leq J} T_{kb}T_{k+1,a} T_{k+2,a} - \sum_{1 \leq c < d \leq J} T_{kd}T_{k+1,d} T_{k+2,c} \right) \right. \\
&\qquad - \left. 4\sum_{1 \leq c < a < b \leq J} T_{kb}T_{k+1,a} T_{k+2,c} \right] \\
&= \frac{(\lambda_k\lambda_{k+1} - f(k))(\lambda_{k+1}\lambda_{k+2} - f(k+1)) - e(k)}{4\lambda_k \lambda_{k+1} (\lambda_{k+1}-1) \lambda_{k+2}} \\
&= \frac{\lambda_k}{4(\lambda_k - 1)} - \frac{f(k+1)}{4(\lambda_{k+1}-1)\lambda_{k+2}} - \frac{f(k)}{4\lambda_k(\lambda_{k+1}-1)} + \frac{f(k)f(k+1) - e(k)}{4\lambda_k \lambda_{k+1} (\lambda_{k+1}-1) \lambda_{k+2}}. \qedhere
\eal
\end{proof}

Using the previous lemmas, we can now compute the mean and asymptotic variance of the number of descents. 

\begin{proposition}\label{descentmeanvariance} 
Let $\sigma$ be a permutation chosen uniformly at random from the double coset in $S_\lambda \bslash S_n / S_\mu$ indexed by $T$, where $I = o(n)$. Then
\bal 
E(\des(\sigma)) &= \frac{n-1}{2} - \frac{1}{2} \sum_{k=1}^{I-1} \frac{f(k)}{\lambda_k \lambda_{k+1}} \\
\Var(\des(\sigma)) &\asymp \frac{n}{12}, 
\eal
where $f(k) := \sum_{1 \leq a < b \leq J} (T_{ka}T_{k+1,b} - T_{k+1,a}T_{kb})$, for $k \in [I-1]$
\end{proposition}

\begin{proof}
The mean immediately follows from linearity of expectation and Lemma \ref{descentprob}, 
\bal 
E(\des(\sigma)) &= \sum_{k=1}^{n-1} E(I_k) \\
&= \sum_{k=1}^I \sum_{i \in L_k^o} P(I_i = 1, i \in L_k^o) + \sum_{k=1}^{I-1} P(I_i = 1, i = l_k^b) \\ 
&= \sum_{k=1}^I \frac{\lambda_k - 1}{2} + \sum_{k=1}^{I-1} \left( \frac{1}{2} - \frac{f(k)}{2\lambda_k\lambda_{k+1}} \right) \\
&=  \frac{n-1}{2} - \frac{1}{2} \sum_{k=1}^{I-1} \frac{f(k)}{\lambda_k \lambda_{k+1}}. 
\eal

Next, we write the variance as
\bal
\Var(\des(\sigma)) = \sum_{i=1}^{n-1} \Var(I_i) + 2\sum_{i=1}^{n-2} \cov(I_i, I_{i+1}) + 2\sum_{|j-i| \geq 2} \cov(I_i, I_j).
\eal

By Lemma \ref{descentprob},
\bal
\sum_{i=1}^{n-1} \Var(I_i) &= \sum_{i=1}^{n-1} E(I_i) - [E(I_i)]^2 = \sum_{k=1}^I \sum_{i \in L_k^o} \frac{1}{4} + \sum_{k=1}^{I-1} \left( \frac{1}{4} - \frac{f(k)^2}{4\lambda_k^2 \lambda_{k+1}^2} \right) \\
&= \frac{n-1}{4} - \frac{1}{4}\sum_{k=1}^{I-1} \frac{f(k)^2}{\lambda_k^2 \lambda_{k+1}^2}.
\eal
The contribution from the sum is a constant times $I$, since
\bal
\left|\frac{f(k)}{\lambda_k \lambda_{k+1}}\right| = \left| \frac{\sum_{1 \leq a < b \leq J} (T_{ka}T_{k+1,b} - T_{k+1,a}T_{kb})}{ \left( \sum_{i=1}^J T_{ki}\right)\left( \sum_{i=1}^J T_{k+1,i} \right) } \right| \leq 1,
\eal
for all $1 \leq k \leq I$. Since $I = o(n)$, this gives
\bal
\sum_{i=1}^{n-1} \Var(I_i) &\asymp \frac{n}{4}.
\eal

Next, by Lemmas \ref{descentprob} and \ref{descentprobtwoconsecutive}, 
\bal
&\sum_{i=1}^{n-2} \cov(I_i, I_{i+1}) = \sum_{i=1}^{n-2} E(I_{i,i+1}) - E(I_i)E(I_{i+1}) \\
&= \sum_{k=1}^I \sum_{i \in L_k^{oo}} \left( \frac{1}{6} - \frac{1}{4} \right) + \sum_{k=1}^{I-1} \left[ \frac{g(k)}{6\lambda_k(\lambda_k -1)\lambda_{k+1}} + \frac{h(k)}{6\lambda_k\lambda_{k+1}(\lambda_{k+1}-1)} \right. \\
&\qquad \left. - \left( \frac{1}{2} - \frac{f(k)}{2\lambda_k\lambda_{k+1}} \right) \right] \\
&= -\frac{n-2I}{12} - \frac{I-1}{2} \\
&\qquad +\frac{1}{6}\sum_{k=1}^{I-1} \left( \frac{g(k)}{\lambda_k(\lambda_k -1)\lambda_{k+1}} + \frac{h(k)}{\lambda_k\lambda_{k+1}(\lambda_{k+1}-1)} + \frac{3f(k)}{\lambda_k\lambda_{k+1}} \right).
\eal
Again, the contribution from the sum is a constant times $I$, since 
\bal
\left|\frac{g(k)}{\lambda_k(\lambda_k -1)\lambda_{k+1}} \right| \leq 1, \quad \text{and} \quad \left| \frac{h(k)}{\lambda_k\lambda_{k+1}(\lambda_{k+1}-1)} \right| \leq 1,
\eal
for all $1 \leq k \leq I$, using the formulas for $g(k)$ and $h(k)$ from Lemma \ref{descentprobtwoconsecutive}. 
This gives
\bal
\sum_{i=1}^{n-2} \cov(I_i, I_{i+1}) &\asymp -\frac{n}{12}
\eal

Finally by Lemma \ref{doubledescentprob}, $E(I_{i,j}) = E(I_i)E(I_j)$ for all $i,j$ with $|j-i| \geq 2$, except for the case when $i = l_k^b$ and $j = l_{k+1}^b$ for $k \in [I-2]$. Then
\bal
\sum_{|j-i| \geq 2} \cov(I_i, I_j) &= \sum_{k=1}^{I-2} \left[\frac{\lambda_k}{4(\lambda_k - 1)} - \frac{f(k+1)}{4(\lambda_{k+1}-1)\lambda_{k+2}} - \frac{f(k)}{4\lambda_k(\lambda_{k+1}-1)} \right. \\
&\qquad + \left. \frac{f(k)f(k+1) - e(k)}{4\lambda_k \lambda_{k+1} (\lambda_{k+1}-1) \lambda_{k+2}}\right].
\eal
The contribution from the sum is a constant times $I$, since
\bal
\left| \frac{f(k)f(k+1) - e(k)}{4\lambda_k \lambda_{k+1} (\lambda_{k+1}-1) \lambda_{k+2}} \right| \leq 1
\eal
for all $k \in [I-2]$. This gives
\bal
\sum_{|j-i| \geq 2} \cov(I_i, I_j) = o(n). 
\eal

Summing over the cases above yields
\bal
\Var(\des(\sigma)) &\asymp \frac{n}{12}. \qedhere
\eal
\end{proof}

\subsection{Central Limit Theorem for Descents} \label{CLTDescents}

In this section we prove the central limit theorem for the number of descents. We start by constructing a dependency graph, $L$, on the family of random variables $\{I_i : 1 \leq i \leq n-1\}$. 

Define the vertex set, $V$, and edge set, $E$, as
\bal
\begin{cases}
V = [n-1] \\
E = \{(i, j) : \text{if $j = i+1$, or if $j=i-1$, or if $i = l_k^b$ and $j = l_{k+1}^b$ for $k \in [I-2]$}\}.
\end{cases}
\eal
The vertices represent each indicator random variable $I_i$. By Lemmas \ref{descentprob}, \ref{descentprobtwoconsecutive}, and \ref{doubledescentprob}, 
the indicators $I_i$ and $I_j$ are dependent if the positions $(i,i+1)$ and $(j,j+1)$ are not disjoint or if $i = l_k^b$ and $j = l_{k+1}^b$ for $k \in [I-2]$, and independent otherwise. 
Thus the graph above is a dependency graph for $\{I_i : 1 \leq i \leq n-1\}$.

With this dependency graph construction, we can now prove Theorem \ref{CLTdescents}. 

\begin{proof}[Proof of Theorem \ref{CLTdescents}]
Consider the family of indicators $\{I_i : 1 \leq i \leq n-1\}$ and the dependency graph, $L$, constructed above. Let $D - 1$ be the maximal degree of $L$. Note that $|V| = n-1$. 
It is clear by construction that $D - 1 = 4$, so that $D = 5$. We also have that $|I_i - E(I_i)| \leq 1$ for all $i \in [n-1]$, and so we set $B := 1$. From Proposition \ref{descentmeanvariance}, 
we have that $\sigma_n^2 := \Var(\des(\sigma)) \asymp \frac{n}{12}$. 

Therefore by Theorem \ref{SteinDependencyGraph}, 
\bal
d_K(W_n, Z) &= O\left( \frac{8\cdot 5^{3/2} \sqrt{n-1}}{\frac{n}{12}} + \frac{8\cdot 5^2 (n-1)}{\left( \frac{n}{12} \right)^{3/2}} \right) = O(n^{-1/2}),
\eal
where $W_n = \frac{\des(\sigma) - E(\des(\sigma))}{\sigma_n}$ and $Z$ is a standard normal random variable. Thus $d_K(W_n, Z) \to 0$ as $n \to \infty$, so that $W_n \xrightarrow{d} Z$. 
\end{proof}


\section{Generalized Descents for Fixed $d$} \label{GeneralizedDescents} 

In this section we prove Theorem \ref{CLTddescents}, the central limit theorem for $d$-descents when $d$ is a fixed positive integer. 

Let $\sigma \in S_n$ be chosen uniformly at random from the double coset in $S_\lambda \bslash S_n / S_\mu$ indexed by $T$. Let $\des_d(\sigma)$ be the number of $d$-descents of $\sigma$. 
Observe that for $n \geq 2$ and $1 \leq d < n$, the number of ordered pairs $(i,j) \in [n]^2$ such that $i < j < i +d$ is $N_n = d(n-d) + \binom{d}{2}$. We order these pairs in lexicographic order and index them by $\{(i_k, j_k)\}_{k=1}^{N_n}$, so that 
\bal
\des_d(\sigma) &= \sum_{k=1}^{N_n} Y_k
\eal
where $Y_k = \I_{\{\sigma(i_k) > \sigma(j_k)\}}$ is the indicator random variable that $\sigma$ has a $d$-descent at $(i_k, j_k)$. 

\subsection{Mean and Variance of $d$-Descents} \label{MeanVarDDescents}

The computation for the mean and variance of $d$-descents use the probabilities computed in Section \ref{MeanVarDescents} for the case of descents. 
Using essentially the same proof, we find that the probability of a $d$-descent is $\frac{1}{2}$ when both indices $i,j$ are in the same block, and is $\frac{1}{2} - \frac{f(k)}{2\lambda_k \lambda_{k+1}}$ when $i$ is in $L_k$ and $j$ is in $L_{k+1}$. 
Since the proof is similar to the descents case, we will not include all the technical details but we will describe the key ideas. 

\begin{proposition}\label{ddescentmeanvariance} 
Let $\sigma$ be a permutation chosen uniformly at random from the double coset in $S_\lambda \bslash S_n / S_\mu$ indexed by $T$, where $I = o(n)$. Let $d \leq \frac{\lambda_I}{2}$ be a fixed constant. Then
\bal 
E(\des_d(\sigma)) &= \frac{nd - \binom{d+1}{2}}{2} - \frac{\binom{d+1}{2}}{2} \sum_{k=1}^{I-1} \frac{f(k)}{\lambda_k \lambda_{k+1}}  \\
\Var(\des(\sigma)) &\asymp nd, 
\eal
where $f(k) := \sum_{1 \leq a < b \leq J} (T_{ka}T_{k+1,b} - T_{k+1,a}T_{kb})$, for $k \in [I-1]$
\end{proposition}

\begin{proof} 
For the expected value, we only need to consider $d$-descents which are fully contained within a $\lambda$-block and those which cross the border between two adjacent $\lambda$-blocks. 
Observe that the number of $d$-descents fully contained in $L_k$ is $d(\lambda_k - d) + \binom{d}{2}$ for all $k \in [I]$, and the number of $d$-descents which cross the border between $L_k$ and $L_{k+1}$ is $\binom{d+1}{2}$ for all $k \in [I-1]$. 
Thus by linearity of expectation and Lemma \ref{descentprob},
\bal
E(\des_d(\sigma)) &= \sum_{k=1}^{N_n} E(Y_k) \\
&= \sum_{k=1}^{I} \frac{d(\lambda_k-d) + \binom{d}{2}}{2} + \sum_{k=1}^{I-1} \binom{d+1}{2} \left( \frac{1}{2} - \frac{f(k)}{2\lambda_k \lambda_{k+1}} \right) \\
&= \frac{nd - \binom{d+1}{2}}{2} - \frac{\binom{d+1}{2}}{2} \sum_{k=1}^{I-1} \frac{f(k)}{\lambda_k \lambda_{k+1}}
\eal

Next, the variance can be written as
\bal
\Var(\des_d(\sigma)) = \sum_{k=1}^{N_n} \Var(Y_k) + 2\sum_{1 \leq k < \ell \leq N_n} \cov(Y_k, Y_\ell). 
\eal

The first sum is computed to be
\bal
\sum_{k=1}^{N_n} \Var(Y_k) &= \sum_{k=1}^{N_n} E(Y_k) - [E(Y_k)]^2 \\
&= \sum_{k=1}^I \frac{d(\lambda_k - d) + \binom{d}{2}}{4} + \sum_{k=1}^{I-1} \binom{d+1}{2}\left( \frac{1}{4} - \frac{f(k)^2}{4\lambda_k^2 \lambda_{k+1}^2} \right) \\
&= \frac{nd - \binom{d+1}{2}}{4} - \frac{\binom{d+1}{2}}{4} \sum_{k=1}^{I-1} \frac{f(k)^2}{\lambda_k^2 \lambda_{k+1}^2} \\
&\asymp \frac{nd}{4},
\eal
since $|\frac{f(k)}{\lambda_k \lambda_{k+1}}| \leq 1$, $I = o(n)$, and $d$ is fixed. 

Now we compute the sum of covariances. We split this into several cases. First consider indicators $Y_k, Y_\ell$ such that $i_k, j_k, i_\ell, j_\ell$ are all in the same block. 
The covariance computation within each block proceeds as in the random permutations case, as computed by Pike \cite{Pik11}, then we sum over all blocks. 
In this case we note that when $\{i_k, j_k\} \cap \{i_\ell, j_\ell\} = \emptyset$, $\cov(Y_k, Y_\ell) = 0$. Thus the only nonzero covariance terms correspond to cases where $i_k = i_\ell$, $j_k = j_\ell$, and $j_k = i_\ell$. 
For the cases $i = i_k = i_\ell$ and $j = j_k = j_\ell$, we have that $E(Y_k Y_\ell) = \frac{1}{3}$, so that $\cov(Y_k, Y_\ell) = \frac{1}{3} - \frac{1}{4} = \frac{1}{12}$. For the case $i_k < j_k = m = i_\ell < j_\ell$, we have that $E(Y_K Y_\ell) = \frac{1}{6}$. 
Summing over all such terms, the contribution to the covariance is of order $nd$. 

Next, consider indicators $Y_k, Y_\ell$ which occur in different blocks with indices fully contained in their respective blocks. That is, $i_k, j_k \in L_a$ and $i_\ell, j_\ell \in L_b$ where $a \neq b$. 
Then the indicators are independent since we have that $E(Y_k Y_\ell) = E(Y_k) E(Y_\ell)$, and so $\cov(Y_k, Y_\ell) = 0$ in this case. 

It remains to consider indicators $Y_k, Y_\ell$ such that at least one of the pairs $(i_k, j_k)$ or $(i_\ell, j_\ell)$ cross between two consecutive blocks. 
If $\{i_k, j_k\} \cap \{i_\ell, j_\ell\} = \emptyset$, we have that $Y_k, Y_\ell$ are independent unless $i_k, i_\ell \in L_a$ and $j_k, j_\ell \in L_b$ or $i_k \in L_a$, $j_k, i_\ell \in L_{a+1}$, and $j_\ell \in L_{a+2}$. 
The sum of covariances of these terms is of order $d^4 I = o(n)$. To see this, note that in this case $E(Y_k Y_\ell) \asymp \frac{1}{4}$, so that $\cov(Y_k, Y_\ell)$ is a constant independent of $n$. 
There are at most $d^4$ such pairs which cross between two blocks $L_a$ and $L_{a+1}$, and at most $d^4$ pairs such that one $d$-descent crosses $L_a, L_{a+1}$ and the other crosses $L_{a+1}, L_{a+2}$. 
Otherwise if $\{i_k, j_k\} \cap \{i_\ell, j_\ell\} \neq \emptyset$, then $Y_k, Y_\ell$ are no longer independent, but by a similar argument as above, the contribution to the sum of covariances is still a constant times $I$. 
Thus the total contribution from these cases is $o(n)$. 

Combining the above cases, we find that $\Var(\des_d(\sigma)) \asymp nd$. 
\end{proof}

We remark that the condition $d \leq \frac{\lambda_I}{2}$ ensures that every $d$-descent occurs at indices which lie at most between two consecutive blocks. 
The regime where $d > \frac{\lambda_I}{2}$ is combinatorially more difficult and so we do not pursue it in the current paper, except for the case of inversions, $d = n-1$, which is considered in the next section. 

\subsection{Central Limit Theorem for $d$-Descents} \label{CLTDDescents}

We now use the method of dependency graphs to prove the central limit theorem for $d$-descents. 

\begin{proof}[Proof of Theorem \ref{CLTddescents}]
Consider the family of indicators $\{Y_k : 1 \leq k \leq N_n\}$ and the dependency graph, $L$, whose vertex set is $[N_n]$ indexed by the indicators $\{Y_k\}$ and edge set $E$ such that $(i,j)$ is an edge if $Y_i$ and $Y_j$ are dependent. 
Note that $|V| = N_n \leq d(n-d) + \binom{d}{2} \leq nd$. 

Let $D - 1$ be the maximal degree of $L$. The maximal degree is attained by indicators which correspond to $d$-descents which cross the border separating $L_k$ and $L_{k+1}$ for some $2 \leq k \leq I-2$. 
That is, indicators $Y_k$ such that $i_k \in L_a$ and $j_k \in L_{a+1}$. 
Observe that such indicators are dependent on all indicators which cross between $L_{a-1}$ and $L_a$, 
all indicators which cross between $L_{k+1}$ and $L_{k+2}$, and all indicators which contain either $i_k$ or $j_k$ as one of the indices which define their $d$-descent. 
Thus $D-1 \leq 4d + 2\binom{d+1}{2}$, so that $D \asymp d^2$. 

Therefore by Theorem \ref{SteinDependencyGraph}, 
\bal
d_K(W_n, Z) &= O\left( \frac{8\cdot (d^2)^{3/2} \sqrt{nd}}{nd} + \frac{8\cdot (d^2)^2 (nd)}{\left( nd\right)^{3/2}} \right) = O(n^{-1/2}),
\eal
where $W_n = \frac{\des_d(\sigma) - E(\des_d(\sigma))}{\sigma_n}$ and $Z$ is a standard normal random variable. Thus $d_K(W_n, Z) \to 0$ as $n \to \infty$, so that $W_n \xrightarrow{d} Z$. 
\end{proof}


\section{Inversions} \label{Inversions} 

Let $\sigma \in S_n$ be chosen uniformly at random from the double coset in $S_\lambda \bslash S_n / S_\mu$ indexed by $T$ and let $\inv(\sigma)$ be the number of inversions of $\sigma$, so that 
\bal
\inv(\sigma) = \sum_{1 \leq i < j \leq n} J_{ij},
\eal
where $J_{ij} = \I_{\{\sigma(i) > \sigma(j)\}}$. 

\subsection{Mean and Variance of Inversions} 

We begin by stating a lemma which gives the probability of an inversion. 

\begin{lemma} \label{inversionprob}
Let $\sigma$ be a permutation chosen uniformly at random from the double coset in $S_\lambda \bslash S_n / S_\mu$ indexed by $T$. Then
\bal
P(I_i = 1) = \begin{cases} 
\frac{1}{2} & \text{if $i,j \in L_k$, for $k \in [I]$},  \\ 
\frac{1}{2} - \frac{f(\ell, r)}{2\lambda_\ell \lambda_r} & \text{if $i \in L_\ell$ and $j \in L_r$, for $1 \leq \ell < r \leq I$},
\end{cases}
\eal
where $f(\ell,r) := \sum_{1 \leq a < b \leq J} (T_{\ell a} T_{r b} - T_{r a} T_{\ell b})$, for $1 \leq \ell < r \leq I$. 
\end{lemma}

The proof follows from essentially the same argument as in Lemma \ref{descentprob}. The main idea is that the probability of an inversion depends only on the location of the indices $i,j$. 
It is $\frac{1}{2}$ if both $i,j$ are in the same block and is a function of $T_{\ell \cdot}, T_{r \cdot}, \lambda_\ell, \lambda_r$ if $i \in L_\ell$ and $j \in L_r$. 

We now state the mean and asymptotic variance of $\inv(\sigma)$. 

\begin{proposition}\label{inversionsmeanvariance}
Let $\sigma$ be a permutation chosen uniformly at random from the double coset in $S_\lambda \bslash S_n / S_\mu$ indexed by $T$, where $I = o(n)$. Then
\bal
&E(\inv(\sigma)) = \frac{n(n-1)}{4} - \frac{1}{2} \sum_{1 \leq \ell < r \leq I} f(\ell, r) \asymp n^2, \\
&\Var(\inv(\sigma)) \asymp n^3,
\eal
where $f(\ell,r) := \sum_{1 \leq a < b \leq J} (T_{\ell a} T_{r b} - T_{r a} T_{\ell b})$. 
\end{proposition}

The proof follows by similar arguments and computations as the descents case, and so we omit it here to preserve the readability of the paper. 

\subsection{Central Limit Theorem for Inversions} \label{CLTInversions} 

The proof of the central limit theorem for the number of inversions follows by a similar argument as in the descents case. First we construct a dependency graph $L$ on the family of indicators $\{J_{ij} : 1 \leq i < j \leq n \}$. 

Define the vertex set, $V$, to be the set of ordered pairs $\{(i,j) : 1 \leq i < j \leq n\}$, and the edge set, $E$, such that there is an edge between the vertices $(i,j)$ and $(i',j')$, with $i \neq i'$ and $j \neq j'$, if either $i = i'$ or $j = j'$.  

We now prove Theorem \ref{CLTinversions} using this dependency graph construction. 

\begin{proof}[Proof of Theorem \ref{CLTinversions}]
Consider the dependency graph, $L$, constructed above for the family of indicators $\{J_{ij} : 1 \leq i < j \leq n\}$. Let $D - 1$ be the maximal degree of $L$. Note that $|V| = \binom{n}{2}$. 
Observe that $D - 1 = 2(n-2) + O(n)$. To see this, consider an indicator $J_{ij}$ such that $i \in L_\ell$ and $j \in L_r$ for $1 \leq \ell < j \leq I$. 
Then $J_{ij}$ is dependent on indicators which contain $i$ or $j$ as an index, or indicators which has exactly one index in $L_\ell$ or exactly one index in $L_r$. This gives $D = O(n)$. 
We set $B := 1$ since $|J_{ij} - E(J_{ij})| \leq 1$ for all $1 \leq i < j \leq n$. From Proposition \ref{inversionsmeanvariance}, we have that $\sigma_n^2 := \Var(\des(\sigma)) \asymp n^3$. 

Therefore by Theorem \ref{SteinDependencyGraph},    
\bal
d_K(W_n, Z) &= O\left(\frac{4n^{3/2}\sqrt{n(n-1)}}{n^3} + \frac{4n^2 n(n-1)}{n^{9/2}}\right) = O(n^{-1/2}),
\eal
where $W_n = \frac{\inv(\sigma) - E(\inv(\sigma))}{\sigma_n}$ and $Z$ is a standard normal random variable. Therefore $d_K(W_n, Z) \to 0$ as $n \to \infty$, from which it follows that $W_n \xrightarrow{d} Z$. 
\end{proof}


\section{Concentration of Measure} \label{Concentration} 

As applications, we use our size-bias coupling and dependency graph constructions from above to obtain concentration of measure results for the number of fixed points, descents, and inversions. 

Our first proposition gives upper and lower tail bounds on the number of fixed points. 

\begin{proposition}
Let $\sigma$ be a permutation chosen uniformly at random from the double coset in $S_\lambda \bslash S_n / S_\mu$ indexed by $T$. Let $W := \fp(\sigma)$ be the number of fixed points of $\sigma$. 
Let $\mu_n := E(W)$ and $\sigma_n^2 := \Var(W)$. Then
\bal
P\left( \frac{W - \mu_n}{\sigma_n} \leq -t \right) &\leq \exp\left( -\frac{\sigma_n^2 t^2}{4\mu_n} \right), \\
P\left( \frac{W - \mu_n}{\sigma_n} \geq t\right) &\leq \exp\left( -\frac{\sigma_n^2 t^2}{4\mu_n + 2\sigma_n t} \right), 
\eal
for all $t > 0$. 
\end{proposition}

\begin{proof}
Let $(W, W^s)$ be the size-bias coupling defined in the proof of Theorem \ref{PoissonLimitTheorem} from Section \ref{PoissonApproximationForFixedPoints}. 
Observe that $|W^s - W| \leq 2$ almost surely, so that we have a bounded size-bias coupling. 
Let $\mu_n := E(W)$ and $\sigma_n^2 := \Var(W)$ be the mean and variance of the number of fixed points computed in Theorems \ref{fixedpointmean} and \ref{fixedpointvariance}, which are both finite and positive. 

Next we have that $W^s \geq W$ almost surely, so that our size-bias coupling is monotone. Moreover, since $W$ is a bounded random variable, its moment generating function $m(s) := E(\exp(sW))$ is finite for all $s \in \R$. 
Therefore applying Theorem \ref{sizebiasconcentration} with $A := \frac{2\mu_n}{\sigma_n^2}$ and $B := \frac{1}{\sigma_n}$ finishes the proof. 
\end{proof}

The next three propositions give matching upper and lower tail bounds on the number of descents, $d$-descents for fixed $d$, and inversions. 

\begin{proposition}
Let $\sigma$ be a permutation chosen uniformly at random from the double coset in $S_\lambda \bslash S_n / S_\mu$ indexed by $T$ and let $\des(\sigma)$ be the number of descents of $\sigma$. 
Let $\mu_n := E(\des(\sigma))$ and $\sigma_n^2 := \Var(\des(\sigma))$. Then
\bal
P\left( \frac{\des(\sigma) - \mu_n}{\sigma_n} \geq t \right) &\leq \exp\left( -\frac{2t^2n}{5} \right), \\
P\left( \frac{\des(\sigma) - \mu_n}{\sigma_n} \leq -t \right) &\leq \exp\left( -\frac{2t^2n}{5} \right),
\eal
for all $t > 0$.
\end{proposition}

\begin{proof}
From the dependency graph for the family of indicators $\{I_i : 1 \leq i \leq n-1\}$ constructed in Section \ref{CLTDescents}, we have that $D = 5$. Moreover $0 \leq I_i \leq 1$ for all $i \in [n-1]$. 
Applying Theorem \ref{dependencyconcentration} and using $\sigma_n^2 \asymp n$ from Proposition \ref{descentmeanvariance} gives
\bal
P\left( \frac{\des(\sigma) - \mu_n}{\sigma_n} \geq t \right) \leq \exp\left( -\frac{2(t\sigma_n)^2}{5\sum_{k=1}^{n-1} 1} \right) \leq \exp\left( -\frac{2t^2n}{5} \right),
\eal
and the same estimate holds for $P\left( \frac{\des(\sigma) - \mu_n}{\sigma_n} \geq -t \right)$. 
\end{proof}

\begin{proposition} 
Let $\sigma$ be a permutation chosen uniformly at random from the double coset in $S_\lambda \bslash S_n / S_\mu$ indexed by $T$ and let $\des_d(\sigma)$ be the number of $d$-descents of $\sigma$, 
where $1 \leq d \leq \frac{\lambda_I}{2}$ is a fixed constant. 
Let $\mu_n := E(\des_d(\sigma))$ and $\sigma_n^2 := \Var(\des_d(\sigma))$. Then
\bal
P\left( \frac{\des_d(\sigma) - \mu_n}{\sigma_n} \geq t \right) &\leq \exp\left( -\frac{2t^2n}{d} \right), \\
P\left( \frac{\des_d(\sigma) - \mu_n}{\sigma_n} \leq -t \right) &\leq \exp\left( -\frac{2t^2n}{d} \right),
\eal
for all $t > 0$.
\end{proposition}

\begin{proof}
From the dependency graph for the family of indicators $\{Y_k : 1 \leq k \leq N_n\}$ constructed in Section \ref{CLTDDescents}, we have that $D \asymp d^2$. Moreover $0 \leq I_i \leq 1$ for all $i \in [n-1]$. 
Applying Theorem \ref{dependencyconcentration} and using $\sigma_n^2 \asymp nd$ from Proposition \ref{descentmeanvariance} gives
\bal
P\left( \frac{\des(\sigma) - \mu_n}{\sigma_n} \geq t \right) \leq \exp\left( -\frac{2(t nd)^2}{d^2 nd} \right) \leq \exp\left( -\frac{2t^2n}{d} \right),
\eal
and the same estimate holds for $P\left( \frac{\des(\sigma) - \mu_n}{\sigma_n} \geq -t \right)$. 
\end{proof}

\begin{proposition} 
Let $\sigma$ be a permutation chosen uniformly at random from the double coset in $S_\lambda \bslash S_n / S_\mu$ indexed by $T$ and let $\inv(\sigma)$ be the number of inversions of $\sigma$. 
Let $\mu_n := E(\des(\sigma))$ and $\sigma_n^2 := \Var(\des(\sigma))$. Then
\bal
P\left( \frac{\inv(\sigma) - \mu_n}{\sigma_n} \geq t \right) &\leq \exp\left( - 2t^2 n^2 \right), \\
P\left( \frac{\inv(\sigma) - \mu_n}{\sigma_n} \leq -t \right) &\leq \exp\left( - 2t^2 n^2 \right),
\eal
for all $t > 0$.
\end{proposition}

\begin{proof}
From the dependency graph for the family of indicators $\{J_{ij} : 1 \leq i < j \leq n\}$ constructed in Section \ref{CLTInversions}, we have that $D \leq n^2$. Moreover $0 \leq J_{ij} \leq 1$ for all $1 \leq i < j \leq n$. 
Applying Theorem \ref{dependencyconcentration} and using $\sigma_n^2 \asymp n^3$ from Proposition \ref{inversionsmeanvariance} gives
\bal
P\left( \frac{\inv(\sigma) - \mu_n}{\sigma_n} \geq t \right) \leq \exp\left( -\frac{2(t n^3) ^2}{n^2 \cdot n^2} \right) \leq \exp\left( - 2t^2 n^2 \right),
\eal
and the same estimate holds for $P\left( \frac{\inv(\sigma) - \mu_n}{\sigma_n} \geq -t \right)$. 
\end{proof}


\section{Final Remarks} \label{FinalRemarks}

\subsection{} A recent line of work in probabilistic group theory is in the study of asymptotic distributions of statistics on permutations chosen uniformly at random from a fixed conjugacy class of $S_n$. 
Fulman initiated the study in \cite{Ful98}, where he proved central limit theorems for descents and major indices in conjugacy classes with large cycles. 
More recently, Kim extended this result to prove a central limit theorem for descents in the conjugacy class of fixed-point free involutions \cite{Kim19}, 
and subsequently, Kim and Lee \cite{KL20} proved a central limit theorem for descents in all conjugacy classes. 
In the setting of all conjugacy classes, Kim and Lee showed that the joint distribution of descents and major indices is asymptotically bivariate normal \cite{KL21}, 
and Fulman, Kim, and Lee showed that the distribution of peaks is asymptotically normal \cite{FKL19}. 

\subsection{} Our paper initiates a parallel line of study on permutations chosen uniformly at random from a fixed parabolic double coset of $S_n$. 
In a work in progress, we investigate the cycle structure in fixed parabolic double cosets. In particular, we study the asymptotic distributions of $k$-cycles, the total number of cycles, and the longest cycle. 

Another interesting statistic is the number of unseparated pairs. This statistic can be used as a test to determine how close to uniformly random a given permutation is. 
A permutation $\sigma$ has an {\em unseparated pair} at position $i$ if $\sigma(i+1) = \sigma(i) + 1$. Let $\up(\sigma)$ denote the number of unseparated pairs of $\sigma$. 
It should be tractable to use Stein's method and size-bias coupling to show that $\up(\sigma)$ is asymptotically Poisson. 

Some other statistics of interest are crossings, excedances, descents plus descents of its inverse, derangements, and pattern occurrences. 

\subsection{} In this paper we were able to obtain central limit theorems, with convergence rates, for descents, $d$-descents for fixed $d \leq \frac{\lambda_I}{2}$, and inversions. 
With a bit more work it should be tractable to extend our methods to the case where $d \leq \frac{\lambda_I}{2}$ and $d$ grows with $\lambda_I$. 

For the regime where $d > \frac{\lambda_I}{2}$, the combinatorics of estimating the variance become much more difficult. 
Moreover, there are much more dependencies between the indicators such that Stein's method no longer holds. 
However, we still expect a central limit theorem to hold under some mild conditions. 
It should be possible to use the method of weighted dependency graphs \cite{Fer18} to prove a central limit theorem, without a convergence rate, in this case. 

\subsection{} Billey et al \cite{BKPST18} studied parabolic double cosets of finite Coxeter groups. A probabilistic study of parabolic double cosets in this more general setting is fully warranted.


\section*{Acknowledgments} 
We thank Jason Fulman for suggesting the problem and for providing many helpful suggestions and references. 
We also thank Gin Park and Mehdi Talbi for several helpful discussions, and Peter Kagey for programming help during the early stages of this project. 
Finally we thank Larry Goldstein for suggesting us to use our constructions to obtain concentration of measure results. 



\Address

\end{document}